\documentclass{article}
\usepackage{amsmath,amsfonts,amssymb,amsthm,epsfig,epstopdf,url,array}
\usepackage{mathtools}
\usepackage{xspace}
\usepackage{xcolor}
\usepackage{marginnote}

\theoremstyle{plain}
\newtheorem{thm}{Theorem}[section]
\newtheorem{lem}[thm]{Lemma}

\newtheorem{cor}[thm]{Corollary}
\newtheorem{dfn}[thm]{Definition}
\newtheorem{ex}[thm]{Example}
\newtheorem{rem}[thm]{Remark}

\newcommand{\eps}{\ensuremath{\varepsilon}}
\newcommand{\EC}{\ensuremath{\mathbf{EC}}\xspace}
\newcommand{\PC}{\ensuremath{\mathbf{PC}}\xspace}
\newcommand{\EQ}{\ensuremath{\mathbf{EQ}}\xspace}
\newcommand{\ECe}{\ensuremath{\EC_\eps}\xspace}
\newcommand{\PCe}{\ensuremath{\PC_\eps}\xspace}
\newcommand{\ECEq}{\ensuremath{\EC^=}\xspace}
\newcommand{\PCEq}{\ensuremath{\PC^=}\xspace}
\newcommand{\ECeEq}{\ensuremath{\ECe^=}\xspace}
\newcommand{\ECeEqOne}{\ensuremath{\ECe^{=_1}}\xspace}
\newcommand{\PCeEq}{\ensuremath{\PCe^=}\xspace}
\newcommand{\lamI}{\ensuremath{\mathbf{\lambda I}}\xspace}
\newcommand{\lamIstar}{\ensuremath{\mathbf{\lambda I}^*}\xspace}
\newcommand{\alllamI}{\ensuremath{\forall\lamI}\xspace}
\newcommand{\alllamIplus}{\ensuremath{\forall\lamI^+}\xspace}
\newcommand{\allelim}{\ensuremath{\forall^-}\xspace}
\newcommand{\allintro}{\ensuremath{\forall^+}\xspace}
\newcommand{\exintro}{\ensuremath{\exists^+}\xspace}
\newcommand{\exelim}{\ensuremath{\exists^-}\xspace}

\newcommand\norm[1]{\left\lVert#1\right\rVert}

\newcommand{\BVar}{\ensuremath{\mathcal{BV}}\xspace}
\newcommand{\OVar}{\ensuremath{\mathcal{FV}}\xspace}
\newcommand{\SVar}{\ensuremath{\mathcal{BV}}\xspace}
\newcommand{\Fun}{\ensuremath{\mathcal{F}}\xspace}
\newcommand{\Pred}{\ensuremath{\mathcal{P}}\xspace}
\newcommand{\Closure}{\ensuremath{\mathcal{CL}}\xspace}

\newcommand{\NQVar}{\ensuremath{\mathrm{FV}}\xspace}

\newcommand{\dgr}{\ensuremath{\mathsf{deg}}\xspace}
\newcommand{\pdgr}{\ensuremath{\mathsf{pd}}\xspace}
\newcommand{\mpdgr}{\ensuremath{\mathsf{mpd}}\xspace}
\newcommand{\arity}{\ensuremath{\mathsf{a}}\xspace}
\newcommand{\marity}{\ensuremath{\mathsf{ma}}\xspace}

\newcommand{\lth}{\ensuremath{\mathsf{len}}\xspace}
\newcommand{\cc}{\ensuremath{\mathsf{cc}}\xspace}
\newcommand{\cce}{\ensuremath{\cc^\eps}\xspace}
\DeclareMathOperator{\depth}{\ensuremath{\mathsf{dp}}\xspace}
\DeclareMathOperator{\frmcomp}{\ensuremath{\mathsf{comp}}\xspace}
\newcommand{\ccEq}{\ensuremath{\cc^=}\xspace}

\newcommand{\rk}{\ensuremath{\mathsf{rk}}\xspace}
\newcommand{\crs}{\ensuremath{\mathcal{CR}}\xspace}

\newcommand{\width}{\ensuremath{\mathsf{wd}}\xspace}
\newcommand{\mwidth}{\ensuremath{\mathsf{mwd}}\xspace}
\newcommand{\widtheq}{\ensuremath{\width^=}\xspace}
\newcommand{\widtheps}{\ensuremath{\width^\eps}\xspace}
\newcommand{\order}{\ensuremath{\mathsf{o}}\xspace}
\newcommand{\matrices}{\ensuremath{\mathsf{m}}\xspace}

\newcommand{\HC}{\ensuremath{\mathrm{HC}}\xspace}

\newcommand{\bicon}{\ensuremath{\mathrm{\circ}}\xspace}
\newcommand{\quant}{\ensuremath{\mathsf{Q}}\xspace}
\newcommand{\aeq}{\ensuremath{\equiv_\alpha}\xspace}

\newcommand*{\impl}{\mathrel{\rightarrow}}
\newcommand*{\tpkt}{\rlap{$\;$.}}
\newcommand*{\tkom}{\rlap{$\;$,}}
\newcommand*{\defsym}{\mathrel{:=}}

\title{The Epsilon Calculus with Equality and Herbrand Complexity}
\author{Kenji Miyamoto and Georg Moser\\\texttt{Kenji.Miyamoto@uibk.ac.at}, \texttt{Georg.Moser@uibk.ac.at}\\
University of Innsbruck}
\date{\today}
\begin{document}
\maketitle
\begin{abstract}
Hilbert's epsilon calculus is an extension of elementary or predicate
calculus by a term-forming operator \(\varepsilon\) and initial
formulas involving such terms.  The fundamental results about the
epsilon calculus are so-called epsilon theorems, which have been
proven by means of the epsilon elimination method.  It is a procedure
of transforming a proof in epsilon calculus into a proof in elementary
or predicate calculus through getting rid of those initial formulas.
One remarkable consequence is a proof of Herbrand's theorem due to
Bernays and Hilbert which comes as a corollary of extended first
epsilon theorem.  The contribution of this paper is the upper and
lower bounds analysis of the length of Herbrand disjunctions in
extended first epsilon theorem for epsilon calculus with equality.  We
also show that the complexity analysis for Herbrand's theorem with
equality is a straightforward consequence of the one for extended
first epsilon theorem without equality due to Moser and Zach.
\end{abstract}

\section{Introduction}
\label{s:intro}
Hilbert's epsilon calculus is an extension of predicate calculus by
the \eps-operator which forms for a formula \(A(x)\) a term \(\eps_x
A(x)\).  This operator is governed by the following two initial formulas:
One is the critical formula
\begin{align*}
A(t) \to A(\eps_x A(x))
\end{align*}
where \(t\) is an arbitrary term, and the other is the \eps-equality
formula
\begin{align*}
\vec{u} = \vec{v} \to \eps_x B(x, \vec{u}) = \eps_x B(x, \vec{v})
\end{align*}
where \(\vec{u}\) and \(\vec{v}\) are sequences of terms \(u_0,
u_1, \ldots, u_{n-1}\) and \(v_0, v_1, \ldots, v_{n-1}\) and \(\vec{u}
= \vec{v}\) stands for the conjunction of \(u_0 = v_0\), \(u_1 =
v_1\), \(\ldots\), and \(u_{n-1} = v_{n-1}\) for an arbitrary positive
natural number \(n\), and the proper subterms of \(\eps_x B(x,
\vec{a})\) are only variables \(\vec{a}\).  Pure epsilon calculus is
an extension of elementary calculus by the \eps-operator and the
critical formula.  The \eps-operator is expressive enough to encode
the existential and universal quantifiers, so that they are definable
as \(\exists x A(x) := A(\eps_x A(x))\) and \(\forall x A(x) :=
A(\eps_x \neg A(x))\) within the epsilon calculus.

The epsilon calculus was originally developed in the context of
Hilbert's program.  Early work in proof theory (before Gentzen)
concentrated on the epsilon calculus, the \eps-elimination method, and
the \eps-substitution method, and those results were carried out by
Bernays~\cite{HilbertBernays39} (see
also~\cite{Zach:2002,Zach:2004,MoserZach06}),
Ackermann~\cite{Ack25,Ack40} (see also~\cite{Moser:2006apal}), and von
Neumann~\cite{Neu27}.  The correct proof of Herbrand's theorem was
first given by means of epsilon calculus~\cite{Buss94}.  The theorem
is commonly stated in a less general way as follows than the original:
If there is a proof of a prenex existential formula \(\exists
{\vec{x}} A(\vec{x})\) for quantifier-free \(A(x)\) in predicate
calculus, there is a proof of \(A(\vec{t}_0) \lor A(\vec{t}_1) \lor
\ldots \lor A(\vec{t}_{k-1})\) for some terms \(\vec{t}_0, \vec{t}_1,
\ldots, \vec{t}_{k-1}\) in elementary calculus.  The epsilon calculus
is of independent and lasting interest, however, and a study from a
computational and proof-theoretic point of view is particularly
worthwhile.

In the course of proving epsilon theorems and Herbrand's theorem, the
\emph{\eps-elimination method} is used to proof-theoretically
transform a proof in epsilon calculus into a proof which is free from
the above mentioned initial formulas.  Assume there is a proof of
\(A(\vec{t}\,)\) in pure epsilon calculus, where \(\vec{t}\) is a
finite sequence of terms possibly with occurrences of \eps-terms, then
the \eps-elimination method generates another proof of the disjunction
\(A(\vec{s}_0) \lor A(\vec{s}_1) \lor \ldots \lor A(\vec{s}_{k-1})\)
in elementary calculus, where \(\vec{s}_0, \vec{s}_1, \ldots,
\vec{s}_{k-1}\) are terms without the \eps-operator.  The disjunction
is a so-called \emph{Herbrand disjunction} for the formula
\(A(\vec{t}\,)\), and the aim of this paper is analyses of the
\emph{Herbrand complexity} which is the length \(k\) of the shortest
Herbrand disjunction for the original formula.  This paper extends the
Herbrand complexity analysis by Moser and Zach~\cite{MoserZach06}.
Their result tells us that the Herbrand Complexity of a formula \(A\)
is based on the proof measure speaking only about the first-order
counterpart of a proof of \(A\).  While they have dealt with the
systems of epsilon calculus without the \eps-equality formula, we
target epsilon calculus with the \eps-equality formula and study the
upper and lower bounds analysis of the Herbrand complexity for the
system with the \eps-equality formula.  Our contribution is divided
into two parts.  The first one is a complexity analysis for Herbrand's
theorem in first-order logic with equality.  In this case, we can
avoid to rely on epsilon calculus with the \eps-equality formula,
hence the result by Moser and Zach is directly applicable.  The second
one is the upper and lower bounds analyses for extended first epsilon
theorem with the \eps-equality formula, where the upper bound analysis
depends on a measure concerning the structure of critical formulas as
well as the measure for first-order ingredients of a proof.

Hilbert's epsilon calculus is primarily a classical formalism, and we
will restrict our attention to classical first-order logic.  For
non-classical approaches to epsilon calculus, see the work of
Bell~\cite{Bell:93a,Bell:93b}, DeVidi~\cite{DeVidi:95},
Fitting~\cite{Fitting:75}, Mostowski~\cite{Mostowski63}, and
Shirai~\cite{Shirai71}.  Our study is also motivated by the recent
renewed interest in the epsilon calculus and the \eps-substitution
method in, e.g., the work of Arai~\cite{Arai01b,Arai:2005},
Avigad~\cite{Avigad01}, Baaz et al.~\cite{BaazLeitschLolic18}, and
Mints et al., \cite{Mints:99,Mints:2003}.  The epsilon calculus also
allows the incorporation of choice construction into
logic~\cite{BlassGurevich:2000:JSL}.  The treatment of eigenvariables
in the context of unsound proofs and its relation to the
epsilon calculus is studied by Aguilera and
Baaz~\cite{AguileraBaaz16}.  On the semantics of epsilon calculus,
see the work of Zach~\cite{Zach17}.

The rest of this paper is organized in the following way.
Section~\ref{s:epsilon-calculus} describes the syntax of epsilon
calculus without the \eps-equality formula, Section~\ref{s:embedding}
shows the embedding lemma which states that predicate calculus is a
subset of pure epsilon calculus without the \eps-equality formula.  A
complexity analyses of Herbrand's theorem for a prenex existential
formula comes as a simple consequence of the lemma.  In
Section~\ref{s:epsilon-calculus-with-equality}, the system is extended
by the \eps-equality formula, which makes the identity schema true
within the system.  Section 5 clarifies the subtlety of complexity
analyses of a system with equality through Yukami's
trick~\cite{Yukami84}.  In
Section~\ref{s:first-and-second-epsilon-theorems} we review first and
second epsilon theorems following the proof by Bernays.
Section~\ref{s:extended-first-epsilon-theorem} and
Section~\ref{s:lower-bounds} are devoted to analysing the upper and
the lower bounds, respectively, where
Section~\ref{s:extended-first-epsilon-theorem} describes our
complexity analysis for extended first epsilon theorem and the upper
bound of the Herbrand complexity.  Section~\ref{s:concl} concludes
this paper.


\section{Epsilon Calculus}
\label{s:epsilon-calculus}
We start from defining terms and formulas of our logic.  As a
convention we assume $x, y, z$ range over a set \SVar of bound
variables, $a, b, c$ over a set \OVar of free variables, $f$ over a
set \Fun of function symbols, and $P, Q, R$ over a set \Pred of
predicate symbols.  The symbol $=$ is reserved for the equality
predicate.  Each function symbol and predicate symbol has an arity,
and \SVar, \OVar, \Fun, \(\{=\}\), and \Pred are disjoint.  We
abbreviate \(t_0, t_1, \ldots, t_{k-1}\) as \(\vec{t}\) and let
\(|\vec{t}\,|\) denote its length \(k\).  We define terms, formulas,
and free variable occurrences.  Notice the difference between free
variables \(\OVar\) and free variable occurrences.
\begin{dfn}[Term and formula]
  \label{d:term-and-formula}
  \emph{Raw terms} $t$ and \emph{raw formulas} $A$, $B$ are
  simultaneously defined as follows.
  \begin{align*}
    & t ::= x \mid a \mid f\vec{t} \mid \eps_x A \\
    & A, B ::= P\vec{t} \mid t=t' \mid \neg A \mid A \to B \mid A \land B
    \mid A \lor B \mid \exists x A \mid \forall x A
  \end{align*}
  Sets of \emph{free variable occurrences} \(\NQVar(t)\) and
  \(\NQVar(A)\) are simultaneously defined, assuming \(\bicon \in
  \{\to, \land, \lor\}\), \(\quant \in \{\exists, \forall
  \}\), and \(z \in \OVar \cup \BVar\).
  \begin{align*}
    & \NQVar(z) := \{z\}, ~\NQVar(f\vec{t}\,) := \NQVar(P\vec{t}\,) := \textstyle \bigcup_{i<|\vec{t}\,|}\NQVar(t_i), ~\NQVar(\neg A) := \NQVar(A),\\
    & \NQVar(\quant x A) := \NQVar(\eps_x A) :=\NQVar(A) \setminus \{x\}, ~ \NQVar(A \bicon B) := \NQVar(A) \cup \NQVar(B).
  \end{align*}
  A raw term \(t\) is a \emph{semiterm} if \(\NQVar(t) \cap
  \BVar \neq \emptyset\), and \(t\) is a \emph{term} if \(\NQVar(t) \cap
  \BVar = \emptyset\).  
  A raw formula \(A\) is a \emph{semiformula} if \(\NQVar(A)
  \cap \BVar \neq \emptyset\), and \(A\) is a \emph{formula} if \(\NQVar(A)
  \cap \BVar = \emptyset\).  
  A (semi)formula and a (semi)term are \emph{quantifier free}
  in case neither \(\forall, \exists\), nor \(\eps\) occurs in them.
\end{dfn}
We abbreviate \(A_0 \wedge A_1 \wedge \ldots \wedge A_n\) as
\(\bigwedge_{i=0}^n A_i\) and also as \(\bigwedge_{i<n+1}A_i\), and
the same convention applies to \(\bigvee\).  Terms of the form
\(\eps_x A\) is called \emph{\eps-terms}.
\begin{dfn}[Substitution]
Assume \(\bicon \in \{\to, \land, \lor\}\) and \(\quant \in \{\exists,
\forall\}\).  For (semi)terms $s, t$, (semi)formulas $A, B$, and
variables \(w,z \in \OVar \cup \BVar\), the substitution \(t\{z/s\}\)
is defined as follows.
\begin{align*}
  & w\{z/s\} := s \quad \text{if \(w \equiv z\)}, \qquad w\{z/s\} := w \quad \text{if \(w \not\equiv z\)},\\
  & (f\vec{t}\,)\{z/s\} := f(\vec{t}\{z/s\}), \qquad (P\vec{t}\,)\{z/s\} := P (\vec{t}\{z/s\}), \\
  & (\neg A) \{z/s\} := \neg (A \{z/s\}), \quad (A \bicon B)\{z/s\} := A \{z/s\} \bicon B\{z/s\},\\
  & (\quant x A)\{z/s\} := \quant x A \text{ and } (\eps_x A)\{z/s\} := \eps_x A \quad \text{if \(x \equiv z\)}, \\
  & (\quant x A)\{z/s\} := \quant {x'} A\{x/x'\}\{z/s\} \text{ and } (\eps_x A)\{z/s\} := \eps_{x'} (A\{x/x'\}\{z/s\}) \quad \text{o.w.},
\end{align*}
where \(\vec{t}\{w/s\} := t_0\{w/s\}, t_1\{w/s\}, \ldots,
t_{n-1}\{w/s\}\) and \(x'\) is fresh.
\end{dfn}
We can write \(A(a)\) for a formula \(A\) with a free variable \(a \in
\NQVar(A)\), and then \(A\{a/t\}\) is abbreviated as \(A(t)\).  This
notation is extended through the vector notation and the simultaneous
substitution.  We employ the same for terms.
\begin{dfn}[$\alpha$-equivalence]
  We define the $\alpha$-equivalence for (semi)terms and
  (semi)formulas as follows.
\begin{align*}
  & x \aeq x, \quad a \aeq a, \quad f\vec{s} \aeq f\vec{t} := \textstyle \bigwedge_{i<|\vec{s}|} s_i\aeq t_i, \\
  & P\vec{s} \aeq P\vec{t} := \textstyle \bigwedge_{i<|\vec{s}|} s_i\aeq t_i, \quad \neg A \aeq \neg B := A \aeq B,\\
  & A \bicon B \aeq A' \bicon B' :=  A \aeq A' \text{ and } B \aeq B'
    \text{ for \(\bicon \in \{\to, \land, \lor\}\)},\\
    & \quant x A(x) \aeq \quant y B(y) := \eps_x A(x) \aeq \eps_y B(y) := A(z) \aeq B(z) \text{ for a fresh \(z\)}.
\end{align*}
\end{dfn}
We also define the term substitution \(t\{s/u\}\) for (semi)terms
\(t,s,u\) through the \(\alpha\)-equivalence instead of the equality
on variables, and the simultaneous substitution.
\begin{dfn}[Set induced by vector]\label{d:vector-set}
For any vector \(\vec{t}\), a set \(\{\vec{t}\,\}\) is defined to be
\(\bigcup_{i<|\vec{t}\,|}\{t_i\}\) via \(\aeq\).  We say a list of
vectors \(\vec{t}_0, \ldots, \vec{t}_{k-1}\) is a split of \(\vec{t}\)
if \(\{\vec{t}_0\} \uplus \cdots \uplus \{\vec{t}_{k-1}\} =
\{\vec{t}\,\}\) and \(\{\vec{t}_i\} \neq \emptyset\) for \(0 \le i < k\).
\end{dfn}
\begin{dfn}[Equality]
  The following formulas are referred to by \EQ.
  \begin{align*}
    & t = t, && s = t \to t=s, && s = t \to t=u \to s=u, \\
    & \vec{s}=\vec{t} \to P\vec{s} \to P\vec{t}, && \vec{s}=\vec{t} \to f\vec{s} = f\vec{t}.
  \end{align*}
\end{dfn}

\begin{dfn}[Elementary calculus and predicate calculus]
  The system of elementary calculus is denoted by \EC, where its
  initial formulas are propositional tautologies and its inference
  rule is modus ponens given as follows.
  \begin{align*}
    \frac{\Gamma \vdash A \qquad \Gamma \vdash A \to B}{\Gamma \vdash B}
  \end{align*}
  The system of first-order predicate calculus is denoted by \PC,
  where the initial formulas are propositional tautologies and the
  following formulas \((\allelim)\) and \((\exintro)\).
  \begin{align}
    &
    \begin{aligned}
      \forall x A(x) \to A (t)
    \end{aligned}
    & \tag{\allelim}
    \\
    &
    \begin{aligned}
      A (t) \to \exists x A(x)
    \end{aligned}
    & \tag{\exintro}
  \end{align}
  The inference rules of \PC are modus ponens and the following
  \((\allintro)\) and \((\exelim)\), where the \emph{eigenvariable}
  \(a\) may not occur in any formula in the axiom \(\Gamma\).
  \begin{align*}
    \frac{\Gamma \vdash A \to B(a)}{\Gamma \vdash A \to \forall
      x B(x)}(\allintro) \qquad \frac{\Gamma \vdash A(a) \to B}{\Gamma \vdash
      \exists x A(x) \to B}(\exelim)
  \end{align*}
\EC and \PC extended by the initial formulas \EQ are called \(\EC +
\EQ\) and \(\PC + \EQ\), respectively.  We alternatively say \ECEq and
\PCEq for them.
\end{dfn}

\begin{dfn}[Epsilon calculus]
  Let a formula of the form
  \begin{align*}
    & A(t) \to A(\eps_x A (x)),
  \end{align*}
  where \(t\) is an arbitrary term and \(A(a)\) is a formula
  containing \(a\), be a \emph{critical formula}, and we define the
  systems \ECe and \PCe by extending \EC and \PC by taking such
  critical formulas as initial formulas.  We say \(\eps_x A (x)\) is
  the \emph{critical \eps-term} of the critical formula and the
  critical formula belongs to \(\eps_x A (x)\).
\end{dfn}
\begin{dfn}[Proof]
  Let \(\mathbf{S}\) be a system which consists of initial
    formulas and inference rules, and assume a set \(\Gamma\) of
    formulas which we call \emph{axioms}.
    A list of formulas is a \emph{proof}
  in \(\mathbf{S}\) from \(\Gamma\), if each formula is an initial
  formula of \(\mathbf{S}\), a formula in \(\Gamma\), or a consequence
  of an inference rule of \(\mathbf{S}\) referring to preceding
  formulas in the proof.  We write \(\mathbf{S}, \Gamma \vdash_\pi A\)
  if and only if a formula \(A\) is the last formula of the proof
  \(\pi\) in system \(\mathbf{S}\) from \(\Gamma\).  We omit
  \(\Gamma\) if it is empty and \(\mathbf{S}\) if there is no
  confusion.  An inference rule consists of one consequence and
  assumptions, and may be displayed using a horizontal line.
\end{dfn}

\begin{dfn}[Languages]
Let the language \(L(\PCeEq)\) be formulas and terms in
Definition~\ref{d:term-and-formula} and the language \(L(\ECeEq)\) be
\(L(\PCeEq)\) without the universal and existential quantifiers.  We
denote by \(L(\PCEq)\) and \(L(\ECEq)\) the sublanguages of \(L(\PCeEq)\)
and \(L(\ECeEq)\) without the \eps-operator, respectively.
Also, \(L(\PCe)\) and \(L(\ECe)\) are the sublanguages of \(L(\PCeEq)\)
and \(L(\ECeEq)\) without the equality symbol, respectively.
Finally, \(L(\PC)\) and \(L(\EC)\) are the sublanguages of \(L(\PCEq)\)
and \(L(\ECEq)\) without the equality symbol, respectively.
\end{dfn}

We give two examples of \ECe-proofs.  These formulas in the examples
are meant to be \eps-calculus versions of the independence of premise
and the drinker's formula.  See also
Example~\ref{e:eps-translation-ip} and
Example~\ref{e:eps-translation-drinkers-formula}.
\begin{ex}
  \label{e:eps-ip-formula-proof}
  Consider the following formula in \(L(\ECe)\).
  \begin{align}
    \label{f:eps-ip-formula}
    (A \to B(\eps_x B(x))) \to A \to B(\eps_x (A \to B(x))).
  \end{align}
  This formula (\ref{f:eps-ip-formula}) is an instance of the critical
  formula, hence a proof of (\ref{f:eps-ip-formula}) is given as follows.
  \begin{align*}
    & (A \to B(\eps_x B(x))) \to A \to B(\eps_x (A \to B(x))) && \text{critical formula}
  \end{align*}
\end{ex}
\begin{ex}
  \label{e:eps-drinkers-formula-proof}
  Consider the following formula in \(L(\ECe)\).
  \begin{align}
    \label{f:eps-drinkers-formula}
    A (\eps_x (A (x) \to A(\eps_y A(y)))) \to A(\eps_y A(y)).
  \end{align}
  An \ECe-proof of this formula (\ref{f:eps-drinkers-formula}) is given as follows.
  \begin{align*}
    & (A(\eps_y A(y))\to A(\eps_y A(y))) \to \\
    & \qquad\qquad A (\eps_x (A (x) \to A(\eps_y A(y)))) \to A(\eps_y A(y)) && \text{critical formula}\\
    & A(\eps_y A(y))\to A(\eps_y A(y)) && \text{propositional tautology} \\
    & A (\eps_x (A (x) \to A(\eps_y A(y)))) \to A(\eps_y A(y)) && \text{modus ponens}
  \end{align*}
\end{ex}

We conclude this section with the following basic results.
\begin{thm}[Deduction theorem] \label{t:deduction}
  Assume \(A\) is a closed formula.  \(\Gamma \vdash A \to B\) iff
  \(\Gamma, A \vdash B\) in \PCe and in \ECe.
\end{thm}

\begin{lem}[Identity schema]
  For any formula \(A(a)\) and terms \(s,t\) in \(L(\mathbf{S})\),
  \(\mathbf{S} \vdash s=t \to A(s) \to A(t)\) holds for
  \(\mathbf{S} \in \{\PCEq, \ECEq\}\).
\end{lem}
\begin{proof}
  By induction on the size of \(A(a)\).
\end{proof}

Note that the above identity schema is not available in \PCEq and
\ECEq, if the language is extended to \(L(\PCeEq)\) and \(L(\ECeEq)\),
respectively.  In Section~\ref{s:epsilon-calculus-with-equality}, we
deal with epsilon calculus with the \eps-equality formula, within
which the identity schema is recovered for \(L(\PCeEq)\) and
\(L(\ECeEq)\).

\section{Embedding Lemma}
\label{s:embedding}

Hilbert introduced the epsilon operator to encode quantifiers, so that
predicate calculus goes to elementary calculus extended with the
critical formula.  This section describes this encoding of \PC within
\ECe.  The idea is to define the quantifiers by \eps-operator as follows,
and recursively apply them.
\begin{align*}
  \exists x A(x) := A(\eps_x A(x)), \qquad \forall x A(x) := A(\eps_x \neg A(x)).
\end{align*}

\begin{dfn}[\eps-translation]
  For a (semi)term \(t\) and a (semi)formula \(A\) we define its \eps-translation
  \(t^\eps\) and \(A^\eps\).  Let \(\vec{t^\eps}\) stand for
  \(t_0^\eps, \ldots, t_{|\vec{t}\,|-1}^\eps\).
  \begin{align*}
    & x^\eps := x, \quad a^\eps := a, \quad (f\vec{t}\,)^\eps := f\vec{t^\eps}, \quad (\eps_x A)^\eps := \eps_x A^\eps, \quad (P\vec{t}\,)^\eps := P\vec{t^\eps}, \\
    & (A \to B)^\eps := A^\eps \to B^\eps, \quad (A \land B)^\eps := A^\eps \land B^\eps, \quad (A \lor B)^\eps := A^\eps \lor B^\eps,\\
    & (\neg A)^\eps := \neg A^\eps, \quad (\exists x A(x))^\eps := A^\eps(\eps_x A^\eps(x)), \quad (\forall x A(x))^\eps := A^\eps(\eps_x \neg A^\eps(x)).
  \end{align*}
\end{dfn}

\begin{ex}
  \label{e:eps-translation-ip}
  Here is the formula of independence of premise in \(L(\PC)\),
  \begin{align*}
    (A \to \exists x B(x)) \to \exists x. A \to B(x),
  \end{align*}
  whose \eps-translation is the formula (\ref{f:eps-ip-formula}) in Example~\ref{e:eps-ip-formula-proof}.
\end{ex}
\begin{ex}
  \label{e:eps-translation-drinkers-formula}
  Here is the drinker's formula in \(L(\PC)\),
  \begin{align*}
    \exists x A (x) \to \forall y A(y),
  \end{align*}
  whose \eps-translation is the formula (\ref{f:eps-drinkers-formula}) in Example~\ref{e:eps-drinkers-formula-proof}.
\end{ex}

\begin{rem}
  The above two examples also show that the \eps-translation of a
  formula which is not provable in intuitionistic logic can be
  provable in \ECe without using any classical propositional
  tautology.
\end{rem}

\begin{dfn}[Regular proof]
  A proof is \emph{regular} if each eigenvariable in the proof is used
  by at most one \allintro or \exelim.
\end{dfn}

\begin{dfn}[Proof size]
   The size \(|\pi|\) of a proof \(\pi\) is the length of the list.
\end{dfn}

If there is a proof, a regular one is always available and whose size
is polynomially bounded to the original non-regular proof.  This fact comes by the following theorem due to Kraj\'i\v{c}ek~\cite{Krajicek94}.

\begin{thm}
  Let \(\norm{\phi}^s\) and \(\norm{\phi}^t\) be the size of the
  smallest sequence-proof and tree-proof of a provable first-order
  formula \(\phi\) in the Hilbert style calculus, respectively.  Then
  there exists a polynomial \(p(x)\) such that \(\norm{\phi}^t \le
  p(\norm{\phi}^s)\) for every provable first-order formula \(\phi\).
\end{thm}

In the rest of this paper we implicitly assume the regularity of
proofs.

\begin{dfn}[Critical count]\label{d:critical-count-1}
  For a proof \(\pi\), we let \(\cc(\pi)\) be the number of critical
  formulas, \allelim, and \exintro in \(\pi\).
\end{dfn}

\begin{lem}[Embedding]
  \label{l:embedding}
  Assume \(\PC + \EQ \vdash_\pi A\) for a formula \(A \in L(\PCEq)\), then
  \(\ECe + \EQ \vdash_\rho A^\eps\) for some \(\rho\) with \(\cc(\rho) \le
  \cc(\pi)\).
\end{lem}

\begin{proof}
  We refer to the formula at line \(k\) in the proof by \(A_k\).  By
  induction on the length \(l:=|\pi|\).  If \(l=0\) it is trivial.  We
  prove the case \(l+1\), making case analysis how the formula
  \(A_{l+1}\) at the line \(l+1\) is derived.  In case it comes by
  modus ponens using \(A_{i}\) and \(A_j\), which is of the form \(A_i
  \to A_{l+1}\), for \(i,j \le l\), by the induction hypotheses there
  are proofs \(\rho_1\) and \(\rho_2\) concluding \({A_i}^\eps\) and
  \({A_i}^\eps \to {A_{l+1}}^\eps\), respectively, hence
  \({A_{l+1}}^\eps\) by modus ponens.  In case \(A_{l+1}\) is derived
  by \allintro, \(A_{l+1}\) is of the form \(\forall x A_i(x)\) for
  \(i \le l\).  As \((\forall x A_i(x))^\eps = {A_i}^\eps(\eps_x \neg
      {A_i}^\eps(x))\), it suffices to substitute \(x\) for \(\eps_x
      \neg {A_i}^\eps(x)\) throughout the proof of \({A_i}^\eps(x)\)
      which is due to the induction hypothesis.  Here we assumed the
      regularity of the proof.  In case \(A_{l+1}\) is derived by
      \exelim, \(A_{l+1}\) is of the form \(\exists x. B(x) \to C\)
      and \(A_i = B(t) \to C\) for \(i \le l\).  As \((\exists x. B(x)
      \to C)^\eps = {B}^\eps(\eps_x({B}^\eps(x) \to C^\eps)) \to
      C^\eps\), it suffices to use modus ponens with a critical
      formula and \(B^\eps(t^\eps) \to C^\eps\), which comes by
      induction hypothesis.  In case \(A_{l+1}\) is by
      \allelim, \(A_{l+1}\) is of the form \(\forall x B(x)
      \to B(t)\) and hence we prove \(B^\eps(\eps_x(\neg B^\eps(x)))
      \to B^\eps(t^\eps)\), whose contrapositive is a critical
      formula.  In case \(A_{l+1}\) is by \exintro, \(A_{l+1}\)
      is of the form \(B(t) \to \exists x B(x)\) and hence we prove
      \(B^\eps(t^\eps) \to B^\eps(\eps_x B^\eps(x))\) that is
      immediate as it is a critical formula.  The rest is the axioms.
      The rest is the cases for propositional tautologies and \EQ,
      which are all trivial.
\end{proof}

\begin{thm}[Herbrand's theorem]
  \label{t:Herbrand}
  Assume \(\exists \vec{x} E(\vec{x})\) is a prenex existential formula in
  \(L(\PCEq)\), namely, \(E(\vec{x})\) is quantifier free, and
  \begin{equation*}
    \PCEq \vdash \exists \vec{x} E(\vec{x}).
  \end{equation*}
  Then there are \eps-free terms \(\vec{t}_0, \vec{t}_1, \ldots,
  \vec{t}_{n}\) for \(n \le 2_{2\cdot\cc(\pi)}^{3 \cdot \cc(\pi)}\)
  such that
  \begin{equation*}
    \ECEq \vdash \bigvee_{i=0}^n E(\vec{t}_i).
  \end{equation*}
\end{thm}
\begin{proof}
  Assume \(\pi\) is the \(\PCEq\)-proof of \(\exists
  \vec{x} E(\vec{x})\).  By means of Lemma~\ref{l:embedding}, there
  is an \ECe-proof \(\rho\) of \((\exists \vec{x} E(\vec{x}))^\eps\),
  which is namely \(E(\vec{e})\) for some \eps-terms \(\vec{e}\), then
  the conclusion follows from extended first epsilon theorem for \ECe
  (cf.~Theorem 16 in \cite{MoserZach06}) with \EQ being propositional
  tautologies.
\end{proof}

\section{Epsilon Calculus with the \eps-Equality Formula}
\label{s:epsilon-calculus-with-equality}
Epsilon calculus with equality was originally introduced by
Hilbert~\cite{HilbertBernays39}.  Assuming \(\eps_x A(x, a)\) is an
\eps-matrix, he formulated the \eps-equality formula as follows.
\begin{align}
  \label{f:eps-equality-Hilbert}
  u=v \to \eps_x A(x, u) = \eps_x A(x, v)
\end{align}
In this section we adopt a variant of the \eps-equality formula which
is given as follows via the vector notation.
\begin{align}
  \label{f:eps-equality}
  \vec{u}=\vec{v} \to \eps_x A(x; \vec{u}) = \eps_x A(x; \vec{v})
\end{align}
Then we define our system of epsilon calculus with equality \ECeEq to
be \(\ECe + \EQ\) extended with the initial formula
(\ref{f:eps-equality}).  The \eps-elimination method and the proofs of
epsilon theorems for \ECeEq can be simpler than the ones for the
original system by Hilbert.  While the notion of \emph{closures} is
crucial in Hilbert and Bernays' work, we do not need this notion in
\ECeEq.  Moreover, concerning the hyperexponential part of the upper
bound analysis of the Herbrand complexity, our result for \ECeEq is
better than the one for the system with (\ref{f:eps-equality-Hilbert}),
as it will be shown in Section~\ref{ss:alt}.

\begin{dfn}[\eps-matrix and semicolon notation]
  An \eps-term \(e\) is an \emph{\eps-matrix} iff each proper subterm
  of \(e\) is a free variable and each free variable in \(e\) occurs
  exactly once.  The \eps-matrix and its immediate subsemiformula can
  be denoted as \(\eps_x A(x; \vec{a})\) and as \(A(x; \vec{a})\),
  respectively, if and only if \(\eps_x A(x, \vec{a})\) is an
  \eps-matrix with its free variables \(\vec{a}\).  We call the free
  variables \(\vec{a}\) of an \eps-matrix \(\eps_x A(x; \vec{a})\) its
  \emph{parameters}.
\end{dfn}

Conventionally, we let \(g\) range over \eps-matrices, possibly with
its parameters \(\vec{a}\) explicitly denoted as \(g(\vec{a})\).  For
any \eps-term, its \eps-matrix is uniquely determined modulo free
variable names.  If \(e\) is a critical \eps-term, the \eps-matrix of
\(e\) is called a critical \eps-matrix.

\begin{dfn}[Arity of \eps-matrix]
  For an \eps-matrix \(\eps_x A(x; \vec{a})\), we define its
  \emph{arity} \(\arity(\eps_x A(x; \vec{a}))\) to be \(|\vec{a}|\).
  Let \(\marity(\pi, r)\) be the maximal arity of critical
  \eps-matrices of rank \(r\) in \(\pi\), and \(\marity(\pi)\) be
  \(\max\{\marity(\pi, r) \mid r \le \rk(\pi)\}\).
\end{dfn}

\begin{lem}\label{l:eps-term-to-matrix}
  If \(e\) is an \eps-term, then \(e \aeq g(\vec{t}\,)\) for some \eps-matrix \(g(\vec{a})\) and \(\vec{t}\).
\end{lem}
\begin{proof}
  Let \(\vec{t}\) be all the immediate subterms of \(e\) and
  \(\vec{a}\) be fresh variables, so that \(e \aeq e(\vec{t}\,)\).
  Then, \(e(\vec{a})\) is the \eps-matrix \(g(\vec{a})\).
\end{proof}

The epsilon calculus with equality by Hilbert and Bernays also employs
the \emph{\eps-equality formula} as an initial formula.

\begin{dfn}[Epsilon calculus with the \eps-equality formula]
\label{d:ECeEq}  
  Let \PCeEq and \ECeEq be \(\PCe + \EQ\) and \(\ECe + \EQ\) extended
  with the following additional initial formula, respectively,
  \begin{align*}
    \vec{u}=\vec{v} \to \eps_x A(x; \vec{u})=\eps_x A(x; \vec{v}),
    & 
  \end{align*}
  where \(\vec{u}\) and \(\vec{v}\) are term vectors of the same
  length as of the parameters \(\vec{a}\) of \eps-matrix
  \(\eps_x A(x; \vec{a})\).
  A formula of the form \(\vec{u}=\vec{v} \to \eps_x A(x;
  \vec{u})=\eps_x A(x; \vec{v})\) is an \emph{\eps-equality formula},
  where \(\eps_x A(x; \vec{u})\) and \(\eps_x A(x; \vec{v})\) are
  called the \emph{critical \eps-terms} of the \eps-equality formula.
  We also say that the \eps-equality formula belongs to \(\eps_x A(x;
  \vec{u})\) and to \(\eps_x A(x; \vec{v})\).
\end{dfn}

According to the semicolon notation, the \eps-equality formula always
belongs to critical \eps-terms \(\eps_x A(x; \vec{u}), \eps_x A(x;
\vec{v})\) which were formed by applying substitutions
\(\{\vec{a}/\vec{u}\}, \{\vec{a}/\vec{v}\}\) to an \eps-matrix
\(\eps_x A(x; \vec{a})\).  The next section details this constraint
from a perspective of complexity analysis.  Due to the \eps-equality
formula, the identity schema is available in \(L(\PCeEq)\).
\begin{lem}[Identity schema]
  Let a formula \(A(a)\) and terms \(s,t\) be in \(L(\PCeEq)\), then
  \(\PCeEq \vdash s=t \to A(s) \to A(t)\).  The same holds for
  \(L(\ECeEq)\) in \ECeEq.
\end{lem}
\begin{proof}
  By induction.
\end{proof}

We further define means of measuring complexity of terms and proofs,
which are used in the next sections to study procedures of eliminating
critical \eps-terms.  The \emph{rank} counts the depth of nesting
\eps-semiterms, while the \emph{degree} counts the depth of nesting
\eps-terms.  Here we suppose that \(\max\{\} = 0\).

\begin{dfn}[Rank] \label{d:rank}
  We define the rank \(\rk(t)\) for a (semi)term \(t\).
  \begin{align*}
    & \rk(a) := \rk(x) := 0, \quad \rk(f\vec{t}\,) := \max\{\rk(t_i) \mid i < |\vec{t}\,|\}, \\
    & \rk(\eps_x A (x)) := \max\{\rk(t) \mid t \text{ subordinates } \eps_x A(x)\} + 1,
  \end{align*}
  where \(t\) subordinates \(\eps_x A(x)\) iff \(x \in \NQVar(t)\) and
  \(t\) is a subsemiterm of \(A(x)\).  We define \(\rk(\pi)\) be
  \(\max\{\rk(e_0), \ldots, \rk(e_{n-1})\}\), where \(\rk(e_0),
  \ldots, \rk(e_{n-1})\) are the critical \eps-terms in \(\pi\).
\end{dfn}
The rank is stable against substitutions.

\begin{lem} \label{l:rank-subst}
  For any terms \(t(\vec{a})\) and \(\vec{u}\), $\rk(t(\vec{a})) =
  \rk(t(\vec{u}))$.
\end{lem}

\begin{proof}
  Comparing with \(t(\vec{a})\), nothing new is subordinating in
  \(t(\vec{u})\) due to the substitution of $\vec{u}$ for \(\vec{a}\),
  hence it is obvious from Definition~\ref{d:rank}.
\end{proof}

\begin{lem} \label{l:rank-eps-matrix}
  For an \eps-matrix \(\eps_x A(x, \vec{b})\), \(\rk(\eps_x A(x,
  \vec{b}))=\rk(A(a, \vec{b}))+1\).
\end{lem}
\begin{proof}
  By induction on the construction of \(\eps_x A(x, \vec{b})\).
\end{proof}

\begin{dfn}[Degree]
  For a (semi)term \(t\), we define its degree
  \(\dgr(t)\).
  \begin{align*}
    & \dgr(a) := \dgr(x) := 0, \quad \dgr(f\vec{t}\,) := \max\{\dgr(t_i) \mid i < |\vec{t}\,|\},\\
    & \dgr(\eps_x A(x)) := \max\{\dgr(t) \mid t \text{ is a subterm of } A(x)\} +1,
  \end{align*}
\end{dfn}

\begin{dfn}[Maximal critical \eps-term]
  Let \emph{maximal critical \eps-terms} of a proof \(\pi\) be the set
  of critical \eps-terms of the greatest degree among the set of
  critical \eps-terms of the greatest rank in a proof \(\pi\).
\end{dfn}

We conclude this section by defining measures for the proof complexity
based on critical \eps-terms, \eps-matrices, critical formulas and
\eps-equality formulas.

\begin{dfn}[Order]
  For a proof \(\pi\), the number of distinct critical \eps-terms of
  rank \(r\) in \(\pi\) is denoted by \(\order(\pi, r)\) which we call
  the \emph{order}, and the number of distinct \eps-matrices is
  defined in the same manner and denoted by \(\matrices(\pi, r)\)
  which we call the \emph{matrix order}.
\end{dfn}

\begin{dfn}[Width]
  Define \(\widtheps(\pi, e)\) and \(\widtheq(\pi, e)\) by the number
  of distinct critical formulas belonging to \(e\) in \(\pi\) and of
  distinct \eps-equality formulas belonging to \(e\) in \(\pi\),
  respectively.  The \emph{width} \(\width(\pi, e)\) is defined to be
  \(\widtheps(\pi, e)+\widtheq(\pi, e)\).  Let \(\vec{e}\) be critical
  \eps-terms in \(\pi\), then the maximal width \(\mwidth(\pi, r)\) is
  defined to be \(\max\{\width(\pi, e_i) \mid \rk(e_i)=r\}\).
\end{dfn}

In order to measure the number of \eps-equality formulas, we replace
the notion of the critical count in
Definition~\ref{d:critical-count-1} by the following one.

\begin{dfn}[Critical count]\label{d:critical-count-2}
  Assume \(\pi\) is a proof in \PCeEq or \ECeEq.  The critical count
  \(\cc(\pi)\) of \(\pi\) is defined to be the sum of the numbers of
  critical formulas, \eps-equality formulas, \allelim, and \exintro in
  \(\pi\).  We let \(\cce(\pi)\) and \(\ccEq(\pi)\) be the numbers of
  critical fomulas and of \eps-equality formulas in \(\pi\),
  respectively.
\end{dfn}

\section{Yukami's Trick}

In this short section, we clarify the need for the restriction to \eps-matrices
in the definition of \eps-equality axioms, cf.~\eqref{f:eps-equality-Hilbert} and~\eqref{f:eps-equality} (see also Definition~\ref{d:ECeEq}).

For the sake of the argument we assume, for the duration of this section only, that the restriction to \eps-matrices is dropped. We focus on the above
formulation of \eps-equality axioms, using vector notation, as expressed in~\eqref{f:eps-equality}. However, the below given argument is equally valid for Hilbert's original definition (if we drop the restriction to \eps-matrices).

We will employ Yukami's trick~\cite{Yukami84} together with folkore results in structural
proof theory~\cite{Buss:1998,Pudlak:1998}. For additional
insight into the proof theoretic strength of applications of
identiy schema, see~\cite{BaazFerm01}.

\begin{thm}[cf.~\cite{Yukami84}]
\label{t:yukami}
Using two instances of the following restricted scheme of identity
\begin{equation}
\label{eq:restr_identity}
t = 0 \impl g(t) = g(0)
\end{equation}
we can uniformly derive $0^k \defsym \overbrace{0 + (0 + \cdots (0+ 0))}^{k{\rm \ times}} = 0$,
from (i) $0 + 0 = 0$, (ii) $\forall x,y,z \ x = y \land y=z \impl x=z$, and
(iii) $\forall x,y \ x + y = y \impl x = 0$.
\end{thm}
\begin{proof}
Let $r_1(0+0) \equiv 0^n + ( 0^{n-1} + \cdots + (0^2 + 0) \ldots )$
where $0+0$ is fully indicated.
Let $r_2[0+0] \equiv 0^{n-1} + ( 0^{n-2} + \cdots + (0^2 + (0 + 0)) \ldots )$,
where $0+0$ in $r_2[0+0]$ refers only to the innermost occuring term $0+0$.
The following equalities can be easily derived
(employing in addition suitable instance of the transitivity axiom (ii)
and axiom (i))
\begin{eqnarray*}
0^n + \overbrace{(0^{n-1} + \cdots + (0^2 + 0))}^{A} & = & \\
0^{n-1} + (0^{n-2} + \cdots + (0 + 0)) & = & \\
\underbrace{0^{n-1} + (0^{n-2} + \cdots + (0^2 + 0))}_{A}
\end{eqnarray*}
if we employ the instances of~\eqref{eq:restr_identity}
$$0+0 = 0 \impl r_1(0+0) = r_1(0)$$
and
$$0+0 = 0 \impl r_2[0+0] = r_2[0]$$
Hence we have derived $r_1(0+0) = r_2[0]$.
Eventually, to obtain the desired result, we apply axiom (iii), as $r_1(0+0) = r_2[0]$
is nothing else than $0^n + r_2[0] = r_2[0]$
($r_2[0]$ is indicated by $A$ above).

Note that the derivation is uniform for any $k \ge 0$: while for any $k$ the proof slightly differs, the number of steps and in particular the critical count is constant.
\end{proof}

\begin{rem}
Using induction one can derive~\eqref{eq:restr_identity} uniformly
from (i) $\forall x \ s(x) \not=0$ and (ii) $\forall x \ x=x$. That is Yukami's trick is available in any suitable rich arithmetical theory.
\end{rem}

The next result clarifies that the restricted identity axioms employed in Yukami's trick are uniformly derivable if no additional restriction on the form of the \eps-terms are enforced in~\eqref{f:eps-equality}. Let ${\EC'}_\eps^=$ denote the extension of the \eps-calculus $\ECe$ with the following axioms to cover \eps-equality:
  \begin{equation*}
    \vec{u}=\vec{v} \to \eps_x A(x; \vec{u}) = \eps_x A(x; \vec{v}) \tkom
  \end{equation*}
  where $\eps_x A(x; \vec{u})$, $\eps_x A(x; \vec{v})$ denote (arbitrary) \eps-terms.
  
  \begin{lem}
    \label{l:yukami}
  The following identity schema, generalising~\eqref{eq:restr_identity}, is derivable in~${\EC'}_\eps^=$:
  \begin{equation*}
    s=t \impl g(s) = g(t) \tkom
  \end{equation*}
  where $g$ is an aribrary term in $L(\ECeEq)$. 
\end{lem}
\begin{proof}
  Let $s,t \in L(\ECeEq)$. Consider the following two critical axioms:
  \begin{equation*}
    g(s)=g(s) \impl \eps_x (x=g(s)) = g(s) \qquad g(t)=g(t) \impl \eps_x (x=g(t)) = g(t) 
  \end{equation*}
  Thus ${\EC'}_\eps^=$ derives (i) $\eps_x (x=g(s)) = g(s)$ as well as (ii) $\eps_x (x=g(t)) = g(t)$.
  We exploit the following \eps-equality axiom in~${\EC'}_\eps^=$
  \begin{equation}
    \label{eq:1}
    s=t \impl \eps_x (x=g(s)) = \eps_x (x=g(t))
    \tpkt
  \end{equation}
  Assuming $s=t$, we can thus derive (within ${\EC'}_\eps^=$) $\eps_x (x=g(s)) = \eps_x (x=g(t))$. Due
  to (i) and (ii) and equality axioms $\EQ$, we thus obtain $g(s) = g(t)$ in ${\EC'}_\eps^=$ as claimed.
  It is important to emphasise, that the $\eps$-term employed in~\eqref{eq:1} is not an $\eps$-matrix.
\end{proof}

Before we can employ Yukami's trick and the above lemma, we need some preparatory definitions and results.

Let $T$ be a theory. We say $T$ \emph{admits Herbrand's theorem} if whenever $T \vdash \exists \vec{x}\ E(\vec{x})$, with $E(\vec{x})$ quantifier-free, then there exists a finite sequence of terms \(\vec{t}_0, \vec{t}_1, \ldots, \vec{t}_{n}\) such that $T \vdash \bigvee_{i=0}^n E(\vec{t}_i)$.
%

Let $T$ be axiomatised by purely universal formulas. Then it is well-known that $T$ admits Herbrand's theorem, cf.~\cite{Buss:1998}. 
Due to Theorem~\ref{t:Herbrand} we can even conclude the existence of a function $f \colon \mathbb{N} \to \mathbb{N}$ such that $n \le f(k)$, where $k$ denotes the critical count of the proof of $T \vdash \exists \vec{x}\ E(\vec{x})$.

The next result improves upon this, in the sense that we also bound the term complexity of the sequence of  terms $\vec{t}_i$ in the critical count. Let $\depth(t)$ denote the \emph{depth} of any term $t \in L(T)$, defined in the usual way. Futher let $\frmcomp(F)$ denote the \emph{formula complexity} of an formula $F$ in the language of $T$. A variant of the following result is due to Krajicek and Pudlak (see~\cite{KP:1988}). 

\begin{thm}
\label{t:term_complexity}  
Suppose $T$ is a universal theory  such that $T \vdash_\pi \exists \vec{x} E(\vec{x})$ so that the underlying equational theory of $T$ (if any), has positive unification type.
Then there exists a primitive recursive function $g$ and a finite sequence of terms $\vec{t}_i$ such that $T \vdash \bigvee_{i=0}^n E(\vec{t}_i)$, where
$n, \depth(t_i) \le g(\cc(\pi),\frmcomp(E(\vec{x})))$.
\end{thm}
\begin{proof}
  Wlog.\ we assume that $T$ is axiomatised by quantifier-free formulas. As $\exists \vec{x}\ E(\vec{x})$ is provable in $T$, there exists a conjunction of (quantifier-free) axioms $Ax$ in $T$ such that
  $\PCEq \vdash Ax \impl \exists \vec{x} E(\vec{x})$.
  By the above, we conclude the existence of terms $\vec{t}_i$ and a primitive recursive function $f\colon \mathbb{N} \to \mathbb{N}$ with $n \le f(k)$, such that
  \begin{equation}
    \label{eq:term_complexity}
    Ax \impl \bigvee_{i=0}^n E(\vec{t}_i) \tkom
  \end{equation}
  is a consequence of the underlying equational theory of $T$, if any. 
  As above, $k$ denotes the critical count of $\pi$.
  It remains to prove the existence of the bounding function $g$, bounding not only the number of terms $t_i$, but also their term depth.

  It is not difficult to formulate the property that the formula~\eqref{eq:term_complexity} is quasi-tautology as a unification
  problem $U$ over the corresponding equational theory, cf.~\cite{KP:1988,Pudlak:1998}.
  As $U$ is solvable there exists a most general unifier $\rho$ such that $(A \impl U)\rho$
  is a quasi-tautology, that is follows from the equational theory of $T$. This follows as unification has a positive unification type by assumption.
  For each term $s$ in the range of $\rho$, $\depth(s)$ is bounded
  by a (monotone) function $h \colon \mathbb{N} \times \mathbb{N} \to \mathbb{N}$, depending only on the input to the unification problem, that is, the length $n$ of the term sequence $\vec{t}_0, \vec{t}_1, \ldots, \vec{t}_{n}$ and the formula complexity of~$E(\vec{x})$.
  
  Finally, it is easy to see how to define a bounding function $g$ such that
  (i) $n \le f(\cc(\pi)) \le g(\cc(\pi),\frmcomp(E(\vec{x})))$ and
  (ii) $\depth(s) \le h(f(\cc(\pi)),\frmcomp(E(\vec{x}))) \le g(\cc(\pi),\frmcomp(E(\vec{x})))$.
\end{proof}

We say a theory $T$ that \emph{admits bounded Herbrand complexity} if whenever $T \vdash_\pi \exists \vec{x}\ E(\vec{x})$, with $E(\vec{x})$ quantifier-free, there exist terms $\vec{t}_i$ such that $T \vdash \bigvee_{i=0}^n E(\vec{t}_i)$ and $n$ is bounded by a function depending only on the formula complexity of $E$ and the number of steps in $\pi$. Note that according to the theorem
any universal theory so that the underlying equational theory of $T$ (if any), has positive unification type admits bounded Herbrand complexity.

The next results clarifies that no theory $T$ can exists that admits bounded Herbrand complexity, while at the same time deriving the assumptions of Theorem~\ref{t:yukami}. 

\begin{cor}
\label{c:yukami}
Let $T$ be a universal theory whose axioms include (i) $0 + 0 = 0$, (ii) $\forall x,y,z \ x = y \land y=z \impl x=z$, and (iii) $\forall x,y \ x + y = y \impl x = 0$ and
let the equational theory of $T$ (if any) be of positive unification type.

Such a theory $T$ cannot admit bounded Herbrand complexity and at the same time derives the restricted identity schema~\eqref{eq:restr_identity}. 
\end{cor}
\begin{proof}
  Suppose to the contrary, a theory $T$ exists whose underlying equational theory has positive unification type. Moreover $T$ admits bounded Herbrand complexity and derives the assumed axioms. 
  Thus due to Theorem~\ref{t:yukami}, $T$ uniformly derive $0^k = 0$ for all $k\in\mathbb{N}$.
  Further for any $k$, there exists a finite set $\{A_1,\dots,A_\ell\}$ of (universally quantified) axioms of $T$, in particular including axioms (i)---(iii), such that 
  \begin{equation*}
    \forall \vec{x} (A_1(\vec{x}) \land \cdots \land A_\ell(\vec{x})) \impl 0^k = 0 \tpkt
  \end{equation*}
  is provable in $\PCEq$ with proofs of constant critical count. Arguing as in the proof of Theorem~\ref{t:term_complexity}, employing the assumption that $T$ that admits bounded
Herbrand complexity, we conclude the existence of terms $\vec{t}_i$ of bounded depth such that
  \begin{equation}
    \label{eq:yukami}
    \bigwedge_{i=0}^n (A_1(\vec{t}_i) \land \cdots \land A_\ell(\vec{t}_i)) \impl 0^k = 0 \tkom
  \end{equation}
  is a quasi-tautology. As the depth of the terms $\vec{t}$ is bounded, while $k$ is unbounded.
  This is absurd. Contradiction to the assumption that $T$ that admits bounded Herbrand complexity.   
\end{proof}

Finally, we arrive at the main result of this section, emphasising the need for restricting the
use of $\eps$-matrices in the epsilon calculus with equality, cf.~Definition~\ref{d:ECeEq}.

\begin{thm}
  Let $T$ be finitely axiomatised by the axioms (i) $0 + 0 = 0$, (ii) $\forall x,y,z \ x = y \land y=z \impl x=z$, and (iii) $\forall x,y \ x + y = y \impl x = 0$ and formalised over ${\EC'}_\eps^=$. Then $T$ cannot admit bounded Herbrand complexity.
\end{thm}
\begin{proof}
  Suppose to the contrary that $T$ admits bounded Herbrand complexity, that is,
  whenever $T \vdash_\pi \exists \vec{x}\ E(\vec{x})$, with $E(\vec{x})$ quantifier-free, there exists terms $\vec{t}_i$ such that $T \vdash \bigvee_{i=0}^n E(\vec{t}_i)$ and $n$ is bounded by a
  function depending only on the formula complexity of $E$ and the number of steps in $\pi$.
  
  Further, as $T$ is axiomatised over ${\EC'}_\eps^=$, $T$ derives the restricted identity schema~\eqref{eq:restr_identity}, cf.~Lemma~\ref{l:yukami}, 
  
  Finally, as the equational theory of $T$ is restriced to syntactic equality, the corresponding unification type is $1$ and the most general unifier of any unification problem is uniquely defined. But this contradicts Corollary~\ref{c:yukami}, stating that no such theory can exists. 
\end{proof}

\section{First and Second Epsilon Theorems}
\label{s:first-and-second-epsilon-theorems}

Assume there is a proof in \ECeEq of a formula \(E\) which is free
from bound variables.  \emph{First epsilon theorem} states that there
is an \ECEq-proof of the same formula \(E\).  The proof of the theorem
is due to the \emph{epsilon elimination method}, which is to replace
critical \eps-terms in a given \ECeEq-proof by other terms,
maintaining the correctness of the proof, so that all critical
formulas and \eps-equality formulas in the proof are eliminated.
Before we go into the general case, we sketch how epsilon elimination
works through very simple examples.  The first part is for the case
only a critical formula belongs to a critical \eps-term, and the
second one is the case both a critical formula and an \eps-equality
formula belong to a critical \eps-term.

\begin{ex}
Consider an \ECe-proof \(\pi\) of the bound variable free formula \(E\)
involving only one critical formula of the following form.
\begin{align*}
  & A(t) \to A(\eps_x A(x))
\end{align*}
We eliminate this critical formula by generating two proofs of \(A(t)
\to E\) and of \(\neg A(t) \to E\) as follows.  Assume \(A(t)\) is an
axiom and replace all occurrences of \(\eps_x A(x)\) throughout the
proof by \(t\), then the above formula goes to \(A(t\{\eps_x A(x)/t\})
\to A(t)\), which is provable as \(A(t)\) is our axiom.  All the other
formulas in \(\pi\) are either propositional tautology or modus
ponens, hence we get a proof of \(A(t) \to E\) without using any
axiom.  On the other hand, assuming \(\neg A(t)\) is an axiom, the
above critical formula is proved due to ex falso quodlibet, hence we
get a proof of \(\neg A(t) \to E\).  Composing the above two proofs by
means of excluded middle \(A(t) \lor \neg A(t)\), we obtain a proof of
\(E\) without a critical formula, hence it is an \EC-proof of \(E\).
\end{ex}
\begin{ex}
Consider another \ECeEq-proof of the bound variable free formula \(E\)
involving one critical formula and one \eps-equality formula of the
following form, assuming \(u\) and \(v\) are \eps-free terms.
\begin{align*}
  & P(t; u) \to P(\eps_x P(x; u); u) \\
  & u=v \to \eps_x P(x; u) = \eps_x P(x; v)
\end{align*}
We first eliminate this \eps-equality formula by generating two proofs
concluding \(u=v \to E\) and \(\neg u=v \to E\), where we still use a
critical formula which is eliminated due to the above argument.
Assume \(u=v\) is an axiom and replace \(\eps_x P(x; u)\) throughout
the proof by \(\eps_x P(x; v)\).  The critical formula goes to \(P(t';
u) \to P(\eps_x P(x; v); u)\), where \(t' := t\{\eps_x P(x; u)/\eps_x
P(x; v)\}\), which is proved by means of another critical formula
\(P(t'; v) \to P(\eps_x P(x; v); v)\) and the identity formulas
\(P(t'; u) \to P(t'; v)\) and \(P(\eps_x P(x; v); v) \to P(\eps_x P(x;
v); u)\) from the axiom \(u=v\).  The \eps-equality formula goes to
\(u=v \to \eps_x P(x; v) = \eps_x P(x; v)\) which is trivially true,
and the identity formulas are trivially true, too.  As all the other
formulas are either propositional tautology or modus ponens, hence we
eliminate the critical formula to have an \ECEq-proof of \(u=v \to
E\).  On the other hand, assuming \(\neg u=v\) is an axiom, the above
\eps-equality formula becomes trivially provable, hence by the
previous argument we get an \ECEq-proof of \(\neg u=v \to E\).
Composing those two proofs using excluded middle \(u=v \lor \neg
u=v\), we get an \ECEq-proof of \(E\) involving a critical formula,
which is eliminated by the previous argument.
\end{ex}
In the rest of this section, we address the general case, namely, the
epsilon elimination method for proofs arbitrarily involving critical
formulas and \eps-equality formulas.  In case there are at least two
different critical \eps-terms in a proof, we have to give a right
order to eliminate the critical \eps-terms, so that the above strategy
works successfully.  The following scenario illustrates an
unsuccessful case.  Consider a proof involving two different critical
formulas.
\begin{align*}
  & A(t, s) \to A(\eps_x A(x, s), s) \\
  & B(\eps_x A(x, s)) \to B(\eps_x A(x, \eps_y B(\eps_x A(x, y))))
\end{align*}
If we first try to eliminate \(\eps_x A(x, s)\) by substituting \(t\), the second formula
goes to the following non-critical formula which is in general not provable.
\begin{align*}
  & B(t) \to B(\eps_x A(x, \eps_y B(\eps_x A(x, y))))
\end{align*}

The following lemmas tell us about substitutions for a non-critical
\eps-term in critical and \eps-equality formulas.  As a consequence,
by eliminating a critical \eps-term \(e\) of the greatest rank,
critical and \eps-equality formulas not belonging to \(e\) are kept to
be critical and \eps-equality formulas after replacing \(e\) by any
term.

\begin{lem} \label{l:non-critical-substitution-1}
  Assume \(e\) is an \eps-term, \(\eta\) is a substitution \(\{e/t\}\)
  where \(t\) is some term, and \(A\) is a critical formula with
  \(\rk(A) \le \rk(e)\).  If \(A\) does not belong to \(e\), \(A\eta\)
  is a critical formula of the rank \(\rk(A)\).
\end{lem}

\begin{proof}
  Let \(A\) be \(B(t; \vec{u}) \to B(\eps_{x} B(x;\vec{u});
  \vec{u})\).  We first show that neither \(t\) nor \(\eps_{x}
  B(x;\vec{u})\) is a proper subterm of \(e\).  If \(e \aeq e'(t)\)
  for some \eps-term \(e'(a)\), \(B(t; \vec{u}) \aeq B'(e'(t);
  \vec{u})\) for some \(B'(a; \vec{b})\).  Then, \(e\) is
  subordinating the critical \eps-term \(\eps_{x} B(x;\vec{u}) \aeq
  \eps_x B'(e'(x); \vec{u})\) and hence \(\rk(e) < \rk(\eps_{x}
  B(x;\vec{u})) = \rk(A)\), which is contradictory.  If \(e \aeq
  e'(\eps_{x} B(x;\vec{u}))\) for some \eps-term \(e'(a)\), \(\dgr(e)
  > \dgr(\eps_{x} B(x;\vec{u}))\) holds, which is contradictory.  As
  the occurrence of \(e\) in \(A\) is as subterms among \(t\) and
  \(\vec{u}\), \(A\eta\) is \(B(t\eta; \vec{u}\eta) \to B(\eps_{x}
  B(x;\vec{u}\eta); \vec{u}\eta)\) which is a critical formula of the
  rank \(\rk(\eps_{x} B(x;\vec{u}\eta)) = \rk(\eps_x B(x; \vec{u}))\).
\end{proof}

\begin{lem} \label{l:non-critical-substitution-2}
  Assume \(e\) is an \eps-term and \(A\) is an \eps-equality formula.
  If \(A\) does not belong to \(e\), the formula \(A\{e/t\}\) for any
  term \(t\) is an \eps-equality formula of rank \(\rk(A)\).
\end{lem}

\begin{proof}
  Let \(A\) be \(\vec{u} = \vec{v} \to \eps_x B(x; \vec{u})=\eps_x
  B(x; \vec{v})\).  As \(e \not \aeq \eps_x B(x; \vec{u})\) and \(e
  \not \aeq \eps_x B(x; \vec{v})\), the substitution can change only
  \(\vec{u}, \vec{v}\), hence \(A\eta\) is \(\vec{u}\eta = \vec{v}\eta
  \to \eps_x B(x; \vec{u}\eta)=\eps_x B(x; \vec{v}\eta)\), which is an
  \eps-equality formula.
\end{proof}

On the other hand, elimination of one critical \eps-term may increase
the number of different critical \eps-terms.  It would be a problem if
it would increase particularly the number of ones of the greatest
rank, because then the termination of our procedure becomes a concern.
Assume a proof involving the following critical formulas belonging to
the two different critical \eps-terms of the greatest rank.
\begin{align*}
  & A(s) \to A(\eps_x A(x)) \\
  & A(t) \to A(\eps_x A(x)) \\
  & B(u, \eps_x A(x)) \to B(\eps_y B(y, \eps_x A(x)), \eps_x A(x))
\end{align*}
If we try to eliminate \(\eps_x A(x)\) first, we afterwards have three
different critical \eps-terms, \(\eps_y B(y, \eps_x A(x))\), \(\eps_y
B(y, s)\), and \(\eps_y B(y, t)\), which is more than the number we
had.  We eliminate a critical \eps-term of the greatest degree among
\eps-terms of the greatest rank, in order not to change any subterm of
critical \eps-terms of the greatest rank.

\begin{lem} \label{l:non-critical-substitution-3}
  Let \(A\) be a critical formula belonging to \(e\) and \(\eta\) be a
  substitution \(\{e'/t\}\) for an \eps-term \(e'\) with \(e \not \aeq
  e'\) and a term \(t\).  If \(\dgr(e) \le \dgr(e')\), \(A\eta\) is a
  critical formula belonging to \(e\).  The same holds for an
  \eps-equality formula \(A\) belonging to \(e_0, e_1\) and a substitution
  \(\{e'/t\}\) for an \eps-term \(e'\) with \(e_i \not \aeq e'\) and
  \(\dgr(e_i) \le \dgr(e)\) for \(i \in \{0,1\}\).
\end{lem}

\begin{proof}
  We prove that \(e'\) does not have an occurrence in \(e\).  Suppose
  to the contrary that \(e'\) has an occurrence in \(e\), \(\dgr(e) >
  \dgr(e')\) which contradicts the assumption \(\dgr(e) \le
  \dgr(e')\).  The case of \eps-equality formulas is trivial.
\end{proof}

By eliminating a maximal critical \eps-term, our \eps-elimination
method illustrated so far successfully decrease the number of
different critical \eps-terms.  By repeating this procedure, we can
eliminate all critical \eps-term of the greatest rank and decrease the
rank of the proof.  Finally, we eliminate all critical \eps-terms to
obtain an \ECEq-proof.

\begin{lem}\label{l:epsilon-elimination-no-eq}
  Assume \(\ECeEq \vdash_\pi E\) for \eps-free formula \(E\) and let
  \(r\) be \(\rk(\pi)\).  If \(e\) is a maximum critical \eps-term in
  \(\pi\) and no \eps-equality formula belongs to \(e\), \(\ECeEq
  \vdash_{\pi_e} E\) for some \(\pi_e\) such that \(\rk(\pi_e)\le r\) and
  \(\order(\pi_e, r) = \order(\pi, r) - 1\).
\end{lem}

\begin{proof}
  Let \(e\) be \(\eps_x A(x; \vec{v})\) and \(w\) be
  \(\width_\pi(e)\).  All the critical formulas belonging to \(e\) in
  \(\pi\) can be listed by \(A(t_i; \vec{v}) \to A(\eps_x A(x;
  \vec{v}); \vec{v})\) for \(i< w\).  We form the proofs \(\bar\pi\)
  and \(\pi_i\) for \(i< w\) such that \(\vdash_{\bar\pi}
  \bigwedge_{i<w} \neg A(t_i; \vec{v}) \to E\) and \(\vdash_{\pi_i}
  A(t_i; \vec{v}) \to E\), from which \(E\) is derivable by
  propositional calculus.  In order to make \(\bar\pi\), we assume
  \(\bigwedge_{i<w} \neg A(t_i; \vec{v})\), then \(A(t_i; \vec{v}) \to
  A(\eps_x A(x; \vec{v}); \vec{v})\) is provable from \(\neg A(t_i;
  \vec{v})\) without those critical formulas belonging to \(e\), hence
  we get \(\bar\pi\) by Theorem~\ref{t:deduction}.  In order to make
  \(\pi_j\), we first replace \(e\) in \(\pi\) by \(t_j\), then prove
  \(A(t_i\{e/t_j\}; \vec{v}) \to A(t_j; \vec{v})\) for each \(i< w\) by
  assuming \(A(t_j; \vec{v})\).  We use modus ponens to the tautology
  \(A(t_j; \vec{v}) \to A(t_i\{e/t_j\}; \vec{v}) \to A(t_j; \vec{v})\)
  and \(A(t_j; \vec{v})\).  Assume \(\pi_e\) is a proof of \(E\)
  obtained by the above procedure, then it remains to prove that
  \(\rk(\pi_e) \le \rk(\pi)\) and \(\order(\pi_e, r) = \order(\pi, r)
  - 1\).  By the construction, critical formulas belonging to \(e\) in
  \(\pi\) don't remain in \(\pi_e\).  For any critical \eps-term
  \(e'\) in \(\pi\), \(\rk(e') < \rk(e)\) implies that critical and
  \eps-equality formulas belonging to \(e'\) in \(\pi\) are critical
  and \eps-equality formulas of the same rank in \(\pi_e\) due to
  Lemma~\ref{l:non-critical-substitution-1} and
  Lemma~\ref{l:non-critical-substitution-2}.  For any critical
  \eps-term \(e'\) in \(\pi\), \(\rk(e') = \rk(e)\) and \(e \not \aeq
  e'\) imply that critical and \eps-equality formulas belonging to
  \(e'\) are not affected by the substitutions for \(e\) due to
  Lemma~\ref{l:non-critical-substitution-3}.  Therefore, the critical
  \eps-term \(e\) does not occur in \(\pi_e\) anymore, \(\rk(\pi_e)
  \le r\), and \(\order(\pi_e, r) = \order(\pi, r) - 1\).
\end{proof}

Note that for \(r' < \rk(\pi)\), \(\order(\pi_e, r')\) may be larger
than \(\order(\pi, r')\) and also maximal degrees of critical and
\eps-equality formulas of rank below \(\rk(\pi)\) in \(\pi_e\) may be
strictly larger than the corresponding original ones in \(\pi\).
Also, \(\width(\pi_e, \rk(\pi))\) may be larger than \(\width(\pi,
\rk(\pi))\), because each premise of critical formulas may be changed
by the substitutions.  As we constantly decrease the order at the
greatest rank, the termination is still guaranteed.  We now study the
epsilon elimination method for \(\ECeEq\).  We use the following
lemma on the identity formula.

\begin{lem} \label{l:aux-1}
  Assume \(A(a; \vec{b})\) has exactly one occurrence of each \(b_i\)
  and for each \(b_i\), if it is a subterm of some term \(t\), \(a \in
  \NQVar(t)\) holds.  For any term \(s\) there exists \(\pi\) and
  \(\ECeEq \vdash_\pi \vec{u}=\vec{v} \to A(s; \vec{u}) \to A(s;
  \vec{v})\), such that \(\cc(\pi) \le |\vec{b}|\cdot\dgr(A(a;
  \vec{b}))\) and \(\rk(\pi) \le \rk(A(a; \vec{b}))\).
\end{lem}

\begin{proof}
  Let \(r_i(a, \vec{b}_i)\) for \(0 \le i < k\) be the immediate
  subterms of \(A(a; \vec{b})\) where the list of vectors \(\vec{b}_0,
  \ldots, \vec{b}_{k-1}\) is a split of \(\vec{b}\) and
  \(\{\vec{b}_i\} \subseteq \NQVar(r_i(a, \vec{b}_i))\).  For some
  \(A^*(\vec{c}\,)\), \(A(a; \vec{b}) \aeq A^*(r_0(a, \vec{b}_0),
  \ldots, r_{k-1}(a, \vec{b}_{k-1}))\).  Assume \(\vec{u}_i\) and
  \(\vec{v}_i\) are the subvectors of \(\vec{u}\) and \(\vec{v}\)
  corresponding to \(\vec{b}_i\).  By induction on the construction of
  \(r_i(a, \vec{b}_i)\), there is a proof \(\pi_i\) of
  \(\vec{u}_i=\vec{v}_i \to r_i(s, \vec{u}_i)=r_i(s, \vec{v}_i)\) such
  that \(\cc(\pi_i) \le |\vec{b}_i| \cdot \dgr(r_i(a, \vec{b}_i))\)
  and \(\rk(\pi_i) \le \rk(r_i(a, \vec{b}_i))\) for each \(i\).  In
  case \(r_i(a, \vec{b}_i)\) is a variable \(b_{ij}\), it is trivial.
  In case \(r_i(a, \vec{b}_i)\) is a function term \(f(r_{i0}(a,
  \vec{b}_{i0}), \ldots, r_{i(l-1)}(a, \vec{b}_{i(l-1)}))\) where the
  list \(\vec{b}_{i0}, \ldots, \vec{b}_{i(l-1)}\) is a split of
  \(\vec{b}_i\), for each \(0 \le j < |\vec{b}_i| =: l\), there is a
  proof \(\pi_{ij}\) of \(\vec{u}_{ij}=\vec{v}_{ij} \to r_{ij}(s,
  \vec{u}_{ij})=r_{ij}(s, \vec{v}_{ij})\) such as \(\cc(\pi_{ij}) \le
  |\vec{b}_{ij}| \cdot \dgr(r_{ij}(a, \vec{b}_{ij}))\) and
  \(\rk(\pi_{ij}) \le r_{ij}(a, \vec{b}_{ij})\) by induction
  hypothesis.  The claim follows by \EQ as \(\cc(\pi_i) \le
  \sum_{j=0}^{l-1} |\vec{b}_{ij}| \cdot \dgr(r_{ij}(a, \vec{b}_{ij}))
  \le l \cdot \dgr(r_{i}(a, \vec{b}_{i}))\) and \(\rk(\pi_i) \le
  \max\{r_{ij}(a, \vec{b}_{ij}) \mid 0 \le j < l\} \le \rk(r_{i}(a,
  \vec{b}_{i}))\).  In case \(r_i(a, \vec{b}_i)\) is an \eps-term
  \(r_i(a, \vec{b}_i) \aeq \eps_y A_i(y; r_{i0}(a, \vec{b}_{i0}),
  \ldots, r_{i(l-1)}(a, \vec{b}_{i(l-1)}))\) where \(\eps_y A_i(y;
  \vec{c}_i)\) is an \eps-matrix, for each \(0 \le j < |\vec{b}_i| =:
  l\), there is a proof \(\pi_{ij}\) of \(\vec{u}_{ij}=\vec{v}_{ij}
  \to r_{ij}(s, \vec{u}_{ij})=r_{ij}(s, \vec{v}_{ij})\) such as
  \(\cc(\pi_{ij}) \le |\vec{b}_{ij}| \cdot \dgr(r_{ij}(a,
  \vec{b}_{ij}))\) and \(\rk(\pi_{ij}) \le \rk(r_{ij}(a,
  \vec{b}_{ij}))\) by induction hypothesis.  The claim follows by
  \eps-equality formulas as \(\cc(\pi_i) \le \left(\sum_{j=0}^{l-1}
  |\vec{b}_{ij}| \cdot \dgr(r_{ij}(a, \vec{b}_{ij}))\right) +1 \le l
  \cdot \dgr(r_{i}(a, \vec{b}_{i}))\) and \(\rk(\pi_i) \le
  \max\{r_{ij}(a, \vec{b}_{ij}) \mid 0 \le j < l\} + 1 = \rk(r_{i}(a,
  \vec{b}_{i}))\).  There is \(\pi\) due to \(\vec{\pi}\) without
  further use of critical nor \eps-equality formulas by induction on
  \(A^*(\vec{c})\), so that \(\cc(\pi) \le \sum_{i=0}^{k-1}
  |\vec{b}_i| \cdot \dgr(r_i(a, \vec{b}_i)) \le |\vec{b}| \cdot
  \dgr(A(a; \vec{b}))\) and \(\rk(\pi)\le \max\{\rk(r_{i}(a,
  \vec{b}_{i})) \mid 0 \le i < k\} \le \rk(A(a; \vec{b}))\).
\end{proof}

We deal with the case both critical formula and \eps-equality formula
belong to the same critical \eps-term.  By the next lemmas, we replace
such a critical formula and an \eps-equality formula, so that we
obtain a proof whose order of the greatest rank is strictly smaller
than the original one.

\begin{lem} \label{l:subst-CA}
  Let \(\eps_x A(x; \vec{b}) =: g(\vec{b})\) be an \eps-matrix, and
  assume \(\dgr(g(\vec{v})) \le \dgr(g(\vec{u}))\).  There is an
  \ECeEq-proof \(\pi\) from the assumption \(\vec{u}=\vec{v}\) such
  that it concludes \(A(t_i; \vec{u}) \to A(\eps_x A(x; \vec{v});
  \vec{u})\) for arbitrary terms \(\vec{t}\) for each \(i <
  |\vec{t}\,|\), and moreover the following conditions hold:
  \(\rk(\pi) = \rk(\eps_x A(x; \vec{b}))\), \(\cc(\pi) \le |\vec{b}|
  \cdot (|\vec{t}\,| + 1)\cdot \dgr(A(a; \vec{b}))+|\vec{t}\,|\), and
  \(\width(\pi, \rk(\pi)) = |\vec{t}\,|\).
\end{lem}

\begin{proof}
  For each \(i\), \(\vec{u}=\vec{v} \vdash A(t_i; \vec{u}) \to A(t_i;
  \vec{v})\) by Lemma~\ref{l:aux-1}.  On the other hand, \(\vdash
  A(t_i; \vec{v}) \to A(\eps_x A(x; \vec{v}), \vec{v})\) as it is a
  critical formula, and also \(\vec{u}=\vec{v} \vdash A(\eps_x A(x;
  \vec{v}), \vec{v}) \to A(\eps_x A(x; \vec{v}), \vec{u})\) by
  Lemma~\ref{l:aux-1}, hence we obtain \(\pi\) using the deduction
  theorem.  As \(\rk(A(a; \vec{b})) \le \rk(\eps_x A(x; \vec{v}))\),
  \(\rk(\pi) = \rk(\eps_x A(x; \vec{v}))\), \(\cc(\pi) \le |\vec{b}|
  \cdot \dgr(A(a; \vec{b})) \cdot |\vec{t}\,| + |\vec{t}\,| +
  |\vec{b}| \cdot \dgr(A(a; \vec{b}))\), and \(\width(\pi, \rk(\pi)) =
  |\vec{t}\,|\) which comes from the number of critical formulas
  belonging to \(g(\vec{b})\).
\end{proof}

\begin{lem} \label{l:subst-eps-=}
  Let \(\eps_x A(x; \vec{b}) =: g(\vec{b})\) be an \eps-matrix, and
  assume \(\dgr(g(\vec{v})) \le \dgr(g(\vec{u}))\) and
  \(\dgr(g(\vec{w})) \le \dgr(g(\vec{u}))\).  There is an \ECeEq proof
  \(\pi\) such that \(\vec{v}=\vec{u} \vdash_\pi \vec{w}=\vec{u} \to
  \eps_x A(x; \vec{v})=\eps_x A(x; \vec{w})\), \(\rk(\pi) = \rk(\eps_x
  A(x; \vec{b}))\), \(\dgr(\pi) \le \dgr(\eps_x A(x; \vec{u}))\), and
  \(\cc(\pi)=1\).
\end{lem}

\begin{proof}
  Assuming \(\vec{v}=\vec{u}\) and \(\vec{w}=\vec{u}\),
  \(\vec{v}=\vec{w}\) holds.  By \eps-equality formula and modus
  ponens, \(\eps_x A(x; \vec{v})=\eps_x A(x; \vec{w})\).  Then we use
  Theorem~\ref{t:deduction}.
\end{proof}

\begin{lem} \label{l:epsilon-elimination-with-eq}
  Assume \(\ECeEq \vdash_\pi E\) for a quantifier-free \(E\) and \(e\)
  is a maximum critical \eps-term in \(\pi\), then \(\ECeEq
  \vdash_{\pi_e} E\) for some \(\pi_e\) where \(\order(\pi_e, r) =
  \order(\pi, r) - 1\) for \(r = \rk(\pi)\) and \(\rk(\pi_e) \le r\).
\end{lem}

\begin{proof}
  If there is no \eps-equality formula belonging to \(e\), we apply
  Lemma~\ref{l:epsilon-elimination-no-eq}.  Otherwise, assume \(e\) is
  of the form \(\eps_x A(x; \vec{u})\), and let \(B_i\) be an
  \eps-equality formula \(\vec{u} = \vec{v}_i \to \eps_x A(x; \vec{u})
  = \eps_x A(x; \vec{v}_i)\) in \(\pi\) for \(i < w\) where \(w :=
  \widtheq(\pi, e)\).  We form the proofs \(\bar \pi\) of
  \((\bigwedge_{i < w} \neg \vec{u} = \vec{v}_i) \to E\) and \(\pi_i\)
  of \(\vec{u} = \vec{v}_i \to E\) for \(i < w\), from which \(E\) is
  derived, in the following procedure.  Assuming \(\bigwedge_{i < w}
  \neg \vec{u} = \vec{v}_i\), all the \eps-equality formulas belonging
  to \(e\) are provable in propositional calculus.  If there is no
  critical formulas belonging to \(e\), we already have \(\bar \pi\),
  and otherwise we further apply
  Lemma~\ref{l:epsilon-elimination-no-eq} to get \(\bar \pi\).
  Concerning \(\pi_i\), assume \(\vec{u} = \vec{v}_i\) and substitute
  \(\eps_x A(x; \vec{v}_i)\) for \(e\) throughout \(\pi\).  Then,
  critical formulas of the form \(A(t; \vec{u}) \to A(\eps_x A(x;
  \vec{u}); \vec{u})\) in \(\pi\) goes to \(A(t; \vec{u}) \to A(\eps_x
  A(x; \vec{v}); \vec{u})\) which is provable by
  Lemma~\ref{l:subst-CA}.  On the other hand, each \eps-equality
  formula \(B_i\) is provable in propositional calculus, and the other
  \eps-equality formulas \(E_j\) for \(j \neq i\) is provable by
  Lemma~\ref{l:subst-eps-=}.  Let \(\pi_e\) be a proof obtained by
  means of propositional calculus from \(\bar \pi\) and \(\vec{\pi}\).
  Since all the critical \eps-term \(e\) in \(\pi\) have been removed
  and the substitutions don't make any other critical \eps-terms
  \(e'\) which is different from \(e\) be same as \(e\), there is no
  critical \eps-term \(e\) in \(\pi_e\).  It remains to prove that the
  obtained proof \(\pi_e\) satisfies \(\order(\pi_e, r) = \order(\pi,
  r) - 1\) for \(r = \rk(\pi)\) and \(\rk(\pi_e) \le r\).  As it is
  apparent for \(\bar \pi\), we consider each \(\pi_i\).  Although
  Lemma~\ref{l:subst-CA} introduces \eps-equality formulas belonging
  to a new \eps-term, those terms are of rank strictly below \(r\).
  Any critical formula of rank \(r\) in each \(\pi_i\) belongs to
  \(\eps_x A(x; \vec{v}_i)\), which is of the same rank \(r\), occurs
  in \(\pi\), and is distinct from \(e\).  The \eps-equality formulas
  of rank \(r\) used in Lemma~\ref{l:subst-eps-=} belong to some
  critical \eps-term of rank \(r\) in \(\pi\) which is different from
  \(e\).  Therefore, \(\order(\pi_e, r) = \order(\pi, r) - 1\) and
  \(\rk(\pi_e) \le r\) hold.
\end{proof}

Repeatedly using the results so far, the rank of the proof is diminished.

\begin{lem} \label{l:decreasing}
  Assume \(\ECeEq \vdash_\pi E\) for a quantifier-free \(E\) and \(e\)
  is a maximum critical \eps-term in \(\pi\), then \(\ECeEq
  \vdash_{\rho} E\) for some \(\rho\) where \(\rk(\rho) <
  \rk(\pi)\).
\end{lem}

\begin{proof}
  We make a sequence of proofs \(\pi_0, \pi_1, \ldots, \pi_n\) for \(n
  = \order(\pi, r)\), where \(\pi_0 := \pi\) and \(r := \rk(\pi)\).
  If no \eps-equality formula belongs to the maximal critical
  \eps-term of \(\pi_{i}\), let \(\pi_{i+1}\) be a proof obtained by
  applying Lemma~\ref{l:epsilon-elimination-no-eq} to \(\pi_i\), and
  otherwise let \(\pi_{i+1}\) be a result of applying
  Lemma~\ref{l:epsilon-elimination-with-eq} to \(\pi_i\).  As in any
  case the order is decreasing, \(\order(\pi_n, r) = 0\) and hence
  \(\rk(\pi_n) < r\), therefore we let \(\rho\) be \(\pi_n\).
\end{proof}

\begin{thm}[First epsilon theorem]
  If \(E\) is a formula in \(L(\ECEq)\) and \(\ECeEq \vdash E\),
  then \(\ECEq \vdash E\).
\end{thm}

\begin{proof}
  Assume \(\ECeEq \vdash_\pi E\).
  We make a sequence of proofs \(\pi_0, \pi_1, \ldots, \pi_r\) for \(r
  = \rk(\pi)\), where \(\pi_0 := \pi\).  In case \(\rk(\pi_{i})=r-i\),
  let \(\pi_{i+1}\) be a proof obtained by applying
  Lemma~\ref{l:decreasing} to \(\pi_i\), and otherwise let
  \(\pi_{i+1}\) be \(\pi_i\).  Then \(\pi_r\) is the \ECEq-proof of \(E\).
\end{proof}

We conclude this section, giving the statement of Second epsilon
Theorem~\cite{HilbertBernays39}.  The proof is given due to the
\eps-elimination method without anything new.

\begin{thm}[Second epsilon theorem]
  If \(A\) is a formula in \(L(\PCEq)\) and \(\PCeEq \vdash A\), then
  \(\PCEq \vdash A\).
\end{thm}
\section{Extended First Epsilon Theorem}
\label{s:extended-first-epsilon-theorem}

In contrast to Section~\ref{s:first-and-second-epsilon-theorems}, we
consider the case the goal formula involves \eps-terms, which leads us
to \emph{extended first epsilon theorem}.  Assume an \ECeEq-proof of a
formula \(E(\vec{s}\,)\) is given, where \(\vec{s}\) may contain
\eps-terms.  Eliminating critical formulas and \eps-equality formulas,
we obtain an \ECEq-proof of a quantifier free formula \(\bigvee_{i <
  n} E(\vec{t}_i)\) for some terms \(\vec{t}_0, \vec{t}_1, \ldots,
\vec{t}_{n-1}\), which is called the \emph{Herbrand disjunction}.  We say that the \emph{Herbrand complexity} of \(E(\vec{s}\,)\) is the smallest length
\(n\) of such a Herbrand disjunction, which is denoted by
\(\HC(E(\vec{s}\,))\).  We firstly study the \eps-elimination method
to prove extended first epsilon theorem without considering the
Herbrand complexity, then the complexity analysis follows to compute
an upper bound of the Herbrand complexity.

\subsection{Proof of extended first epsilon theorem}

We start from describing how to eliminate maximal critical \eps-terms.
In case there is a critical formula belonging to a critical \eps-term
of the maximal rank, we follow the \eps-elimination method described
in Section~\ref{s:first-and-second-epsilon-theorems} for
\ECeEq~\cite{HilbertBernays39}.  Otherwise, only \eps-equality
formulas belong to critical \eps-terms of the maximal rank.  Then we
substitute a function symbol for the \eps-matrix corresponding to the
\eps-term, in order to find a better upper bound of the Herbrand
complexity.

The following lemma is for the first case, i.e.~a critical formula is
involved.

\begin{lem}\label{l:ext-first-eps-eq-1}
  Assume \(\ECeEq \vdash_\pi E(\vec{s}\,)\) for a formula \(E(\vec{a}) \in
  L(\ECEq)\) and terms \(\vec{s} \in L(\ECeEq)\).  If \(e\) is the maximal
  \eps-term of \(\pi\), there is a proof \(\pi_e\) with \(\rk(\pi_e)
  \le \rk(\pi)\) and \(\order(\pi_e, r) < \order(\pi, r)\) such that
  \(\ECeEq \vdash_{\pi_e} \bigvee_{i=0}^{\width_\pi(e)} E(\vec{s}_i)\)
  for \(\vec{s}_{i} \in L(\ECeEq)\).
\end{lem}

\begin{proof}
  Assume \(\eps_x A(x; \vec{u}) =: e\) is the maximal critical
  \eps-term in \(\pi\) and \(\vec{v}_i = \vec{u} \to \eps_x A(x;
  \vec{v}_i)=\eps_x A(x; \vec{u})\) for \(0 \le i < \widtheq_\pi(e) =:
  l\) and \(A(t_j; \vec{u}) \to A(\eps_x A(x; \vec{u}); \vec{u})\) for
  \(0 \le j < \widtheps_\pi(e) =: m\) occur in \(\pi\) as
  \eps-equality and critical formulas, resp.  We form proofs
  \(\bar{\pi}\) of \((\bigwedge_{i=0}^{l-1} \vec{v}_i \neq \vec{u})
  \to \bigvee_{i=0}^{m} E(\vec{s}_i)\) where \(\vec{s}_0 := \vec{s}\)
  and \(\vec{s}_j\) for \(0 < j \le m\) is a result of replacing \(e\)
  in \(\vec{s}\) by \(t_{j-1}\), and \(\pi_i\) of \(\vec{v}_i =
  \vec{u} \to E(\vec{s}_{m+i+1})\), where \(\vec{s}_{m+i+1}\) is a
  result of replacing \(e\) in \(\vec{s}\) by \(\eps_x A(x;
  \vec{v}_i)\).  To get the former proof, assume
  \(\bigwedge_{i=0}^{l-1} \vec{v}_i \neq \vec{u}\), then all the
  \eps-formulas belonging to \(e\) are eliminated due to ex falso
  quodlibet.  By \eps-elimination for \ECe, remaining critical
  formulas belonging to \(e\) are eliminated and we get a proof of
  \(\bigvee_{j=0}^{m} (\bigwedge_{i=0}^{l-1} \vec{v}_i \neq \vec{u}
  \to E(\vec{s}_j))\).  As \(e\) is maximal \eps-term, \(\vec{u}\) and
  \(\vec{v}_i\) for \(0 \le i < l\) are not affected by the
  substitution, hence \(\bigwedge_{i=0}^{l-1} \vec{v}_i \neq \vec{u}
  \to \bigvee_{j=0}^{m} E(\vec{s}_j)\) follows.  To get the latter
  proof for \(j\) such that \(0 \le j < l\), assume \(\vec{v}_j =
  \vec{u}\) and substitute \(\eps_x A(x; \vec{v}_j)\) in \(\pi\) for
  \(e\).  Due to Lemma~\ref{l:subst-eps-=}, the modified \eps-equality
  formulas in \(\pi\) are provable after the substitution.  Due to
  Lemma~\ref{l:subst-CA}, the changed critical formulas which belonged
  to \(e\) in \(\pi\) are provable after the substitution.  We obtain
  \(\pi_e\) combining \(\bar{\pi}\) and \(\vec{\pi}\), where there is
  no critical \eps-term belonged to by a critical formula nor an
  \eps-equality formula, \(\rk(\pi_e) \le \rk(\pi)\), and there is no
  new critical \eps-term of rank \(\rk(\pi)\) belonged by an
  \eps-equality formula.
\end{proof}

Repeating the above lemma, we obtain a proof of a strictly smaller
rank which concludes a Herbrand disjunction.

\begin{dfn}[Critical ranks]
  For a proof \(\pi\), define the set \(\crs(\pi)\) of the ranks of
  critical formulas by \(\{\rk(e) \mid \text{a critical formula
    belongs to \(e\)}\}\).
\end{dfn}

\begin{lem}\label{l:ext-first-eps-eq-4}
  Assume \(\ECeEq \vdash_\pi E(\vec{s}\,)\) for \(E(\vec{a}) \in
  L(\ECEq)\) and \(\vec{s} \in L(\ECeEq)\).  If there is a critical
  formula of the rank \(\rk(\pi) =: r\), there is a proof \(\pi'\)
  such that \(\rk(\pi') < r\), \(\crs(\pi')=\crs(\pi) \setminus
  \{r\}\), and \(\ECeEq \vdash_{\pi'} \bigvee_{i=0}^n E(\vec{s}_i)\)
  where \(\vec{s}_0, \ldots, \vec{s}_{n} \in L(\ECeEq)\) for some \(n\).
\end{lem}

\begin{proof}
  At most \(\order(\pi, r)\) times applications of
  Lemma~\ref{l:ext-first-eps-eq-1} eliminate all the critical
  \eps-terms of rank \(\rk(\pi)\) in \(\pi\).
\end{proof}

The \emph{function symbol substitution} is applicable for eliminating
\eps-equality formulas, provided no critical formula belongs to any
\eps-term of those \eps-equality formulas.

\begin{dfn}[Function symbol substitution]
  Assume \(e\) is a critical \eps-term of the form \(g(\vec{u})\),
  where \(g\) is the \eps-matrix of \(e\), and let \(f\) be a function
  symbol of the arity \(|\vec{u}|\) which is uniquely assigned to
  \(g\).  The substitution \(\{g(\vec{u})/f(\vec{u})\}\) is
  the \emph{function symbol substitution} for \(e\).
\end{dfn}

We are particularly interested in using the function symbol
substitution for the case only \eps-equality formulas are of the
maximal rank.

\begin{lem}
  \label{l:ext-first-eps-only-eq}
  Assume \(\ECeEq \vdash_\pi E(\vec{e}\,)\) for \(E(\vec{a}) \in
  L(\ECEq)\) and for \eps-terms \(\vec{e}\).  If only \eps-equality
  formula is of rank \(\rk(\pi)\), there is an \ECeEq-proof \(\pi'\)
  of \(E(\vec{t}\,)\) for some terms \(\vec{t}\) such that \(\rk(\pi')
  < \rk(\pi)\), \(\crs(\pi')=\crs(\pi)\), and \(\cc(\pi') <
  \cc(\pi)\).  Moreover for any \(r' < \rk(\pi)\), \(\order(\pi', r')
  \le \order(\pi, r')\), and \(\matrices(\pi', r') = \matrices(\pi,
  r')\).
\end{lem}

\begin{proof}
  Let \(r\) be the rank \(\rk(\pi)\).  We repeatedly replace the
  maximal critical \eps-term of rank \(r\) through the corresponding
  function symbol substitution.  After \(\order(\pi, r)\) times of
  replacements, there is no more critical \eps-term of rank \(r\) and
  the process terminates.  After the substitutions, each \eps-equality
  formula of rank \(r\) in \(\pi\) is an identity formula for a
  function symbol \(f\).  Each critical (and \eps-equality
  respectively) formula of a rank strictly smaller than \(\rk(\pi)\)
  is another critical (and \eps-equality respectively) formula which
  belongs to the same critical \eps-term after the substitution due to
  Lemma~\ref{l:non-critical-substitution-3}, hence what we obtained is
  for some terms \(\vec{t}\) an \ECeEq-proof \(\pi'\) of
  \(E(\vec{t}\,)\) with the stated conditions.
\end{proof}

\begin{lem}\label{l:ext-first-eps-eq-5}
  Assume \(\ECeEq \vdash_\pi E(\vec{s}\,)\) for \(E(\vec{a}) \in
  L(\ECEq)\) and \(\vec{s} \in L(\ECeEq)\).  There is an \ECeEq-proof
  \(\pi'\) of the formula \(E(\vec{t}\,)\) for \(\vec{t} \in L(\ECeEq)\)
  such that \(\order(\pi') \le \order(\pi)\) and
  \(\crs(\pi')=\crs(\pi)\).  Moreover, \(\crs(\pi)\) is empty or
  \(\rk(\pi') \in \crs(\pi')\).
\end{lem}

\begin{proof}
  Assume \(\rk(\pi) \not \in \crs(\pi)\), namely, there are \(r_1,
  r_2, \ldots, r_n > \max(\crs(\pi))\) and for each \(r_i\), there are
  \eps-equality formulas of rank \(r_i\) in \(\pi\).  We apply
  Lemma~\ref{l:ext-first-eps-only-eq} for \(n\)-times.
\end{proof}

\begin{rem}
  Even in case there are both critical and \eps-equality formulas of
  the maximal rank, it is still possible to make use of the function
  symbol substitution to eliminate \eps-equality formulas, provided
  there is no critical formula belonging to any critical \eps-term of
  those \eps-equality formulas.  During the substitution process, it
  happens to have non-proofs because formulas of the form
  \(\vec{u}=\vec{v} \to f(\vec{u})=\eps_x A(x; \vec{v})\) are present,
  but eventually these formulas will be of the form \(\vec{u}=\vec{v}
  \to f(\vec{u})=f(\vec{v})\), an instance of identity formula of a
  function symbol.
\end{rem}

\begin{thm}[Extended first epsilon theorem for \ECeEq]
  \label{t:ext-first-eps-eq}
  Assume \(\ECeEq \vdash_\pi E(\vec{s}\,)\) for \(E(\vec{a}) \in
  L(\ECEq)\) and \(\vec{s} \in L(\ECeEq)\).  There is a proof \(\pi'\)
  such that \(\ECEq \vdash_{\pi'} \bigvee_{i=0}^n E(\vec{s}_i)\) for
  some \(n\) and \(\vec{s}_0, \ldots, \vec{s}_{n} \in L(\EC)\).
\end{thm}

\begin{proof}
  We make a sequence of proofs \(\pi_0, \pi_1, \ldots, \pi_{m}\) for
  \(m=|\crs(\pi)|\) in the following way.  Let \(\pi_0\) be a result
  of applying Lemma~\ref{l:ext-first-eps-eq-5} to \(\pi\).  If
  \(\crs(\pi_i)\) is not empty, let \(\pi_{i+1}\) be a proof obtained
  by applying Lemma~\ref{l:ext-first-eps-eq-4} and then
  Lemma~\ref{l:ext-first-eps-eq-5} to \(\pi_i\).  For each \(i\),
  \(\crs(\pi_{i+1}) = \crs(\pi_{i}) \setminus \{\rk(\pi_i)\}\).  Since
  \(\rk(\pi_{m}) = 0\), \(\pi_{m}\) is an \ECEq-proof.  Remaining
  occurrences of \eps-terms may be replaced by free variables.
\end{proof}

\subsection{Complexity analysis}
\label{s:extended-epsilon-calculus-complexity-analysis}

We make a detailed analysis on the proofs of the previous subsection,
in order to compute the numerical bound of the length of the
disjunction in Theorem~\ref{t:ext-first-eps-eq}.  To do so, we
consider the \emph{property degree} as a means of measuring the
complexity of a critical formula.  Given a critical formula \(A(t) \to
A(\eps_x A(x))\), we can determine the formula \(A(a)\).  The property
degree is to count the number of \eps-terms with a free variable
occurrence of \(a\).  If \(A(a)\) is of the form \(A'(a; b)\), this
number tells us at most how many \eps-equality formulas are needed to
prove the identity formula \(u=v \to A'(a; u) \to A'(a; v)\).
We define the \emph{property degree} and the \emph{maximal property
  degree} for an \eps-term as follows.

\begin{dfn}[Property degree]
  For an \eps-term \(e\), the \emph{property degree} \(\pdgr(e)\) is
  defined to be \(\max\{\dgr(t) \mid \text{\(t\) subordinates
    \(e\)}\}\).  The \emph{maximal property degree} \(\mpdgr(\pi, r)\)
  of a proof \(\pi\) of rank \(r\) is defined to be
  $$
  \max\{\pdgr(e) \mid
  \text{\(e\) of rank \(r\) is belonged to by an \eps-equality formula
    in \(\pi\)}\}.
  $$
  Also \(\mpdgr(\pi)\) is defined to be
  $$
  \max\{\pdgr(e) \mid \text{\(e\) belonged to by an \eps-equality
    formula in \(\pi\)}\}.$$
\end{dfn}

We give the following results concerning the property degree.

\begin{lem}
  For any \(\vec{u}, \vec{v}\) and \eps-matrix \(g(\vec{b})\),
  \(\pdgr(g(\vec{u}))=\pdgr(g(\vec{v}))\).  For a given proof
  \(\pi\), the set of critical \eps-matrices which belong to critical
  formulas does not increase through the \eps-elimination.
\end{lem}

\begin{proof}
  The first part is trivial.  Concerning the second part, we consider
  just Lemma~\ref{l:subst-CA}, because a new critical formula is
  introduced only if a critical term belongs to both critical formulas
  and \eps-equality formulas at the same time.  The critical formulas
  introduced in the Lemma belongs to an \eps-term of an \eps-matrix in
  the original proof, hence the claim holds.  The third part is trivial
  since the epsilon elimination method does not add a new \eps-matrix.
\end{proof}

Applying the epsilon elimination method, the maximal property degree
and the maximal arity are weakly decreasing.  Therefore, we can
compute them at the very beginning and keep referring to them as the
upper bounds of the property degrees and the arities through the whole
elimination procedure.

\begin{lem}
  The epsilon elimination method does not increase the maximal property
  degree nor the maximal arity.
\end{lem}
\begin{proof}
Notice that the method does not introduce any new \eps-matrix.
\end{proof}

As a consequence, the upper bounds of the critical counts in
Lemma~\ref{l:aux-1} and Lemma~\ref{l:subst-CA} depend only on an
initial proof.

\begin{lem}[Elimination for critical formulas]
  \label{l:critical-eps-term-elim-ca}
  Let \(E(\vec{a})\) be a formula in \(L(\ECEq)\), \(\vec{s}\) terms
  of \(L(\ECeEq)\), and \(\pi\) an \(\ECeEq\)-proof of \(E(\vec{s})\)
  where its maximal \eps-term \(e := \eps_x A(x; \vec{v})\) is
  belonging only to a critical formula.  There are terms \(\vec{s}_i\)
  such that each of them is in \(L(\ECeEq)\) and an \(\ECeEq\)-proof
  \(\pi_e\) of \(\bigvee_{i=0}^{w}E(\vec{s}_i)\) for
  \(w=\width_\pi(e)\).  Moreover, \(\cc(\pi_e) \le \cc(\pi) \cdot
  (w+1)\) and \(\mwidth(\pi_e, r) \le \mwidth(\pi, r) \cdot (w+1) \le
  (\mwidth(\pi, r))^2\) hold for \(r=\rk(\pi)\).
\end{lem}
\begin{proof}
  Straightforward from Lemma 21 by Moser and Zach~\cite{MoserZach06}.
\end{proof}

\begin{lem}[Elimination for \eps-equality formulas]
  \label{l:critical-eps-term-elim-eqeps}
  Let \(E(\vec{a})\) be a formula in \(L(\ECEq)\), \(\vec{s}\) terms
  of \(L(\ECeEq)\), and \(\pi\) an \ECeEq-proof of \(E(\vec{s})\)
  where its maximal \eps-term \(e := \eps_x A(x; \vec{v})\) is
  belonging only to \eps-equality formulas.  There are terms
  \(\vec{s}_i\) such that each of them is in \(L(\ECeEq)\) and an
  \ECeEq-proof \(\pi_e\) of \(\bigvee_{i=0}^{w}E(\vec{s}_i)\) for
  \(w=\width_\pi(e)\).  Moreover, \(\cc(\pi_e) \le k \cdot (w+1) -
  2\cdot w\) and \(\mwidth(\pi_e, r) \le \mwidth(\pi, r) \cdot (w+1)-
  2 \cdot w\) hold for \(r=\rk(\pi)\) and \(k=\cc(\pi)\).
\end{lem}

\begin{proof}
  Assume \eps-equality formulas belonging to \(e\) are \(\vec{u}_i =
  \vec{v} \to \eps_xA(x; \vec{u}_i) = \eps_x A(x; \vec{v}\,)\) for \(0
  \le i < w\).  There is a proof \(\bar{\pi}\) of \((\bigwedge_{i=0}^k
  \vec{u}_i = \vec{v}) \to E(\vec{s}\,)\) satisfying \(\cc(\bar{\pi})
  \le k - w\) and \(\mwidth(\bar{\pi}, r) \le \mwidth(\pi, r) - w\).
  On the other hand there are proofs \(\pi_i\) of \(\vec{u}_i =
  \vec{v} \to \eps_x A(x; \vec{u}_i) = \eps_x A(x; \vec{v})\) for \(0
  \le i < w\) with \(\cc(\pi_i) \le k - 1\) and \(\mwidth(\pi_i, r)
  \le \mwidth(\pi, r) - 1\).  In order to get \(\pi_i\) we replace all
  \(e\) in \(\pi\) by \(\eps_x A(x, \vec{u}_i)\), and give a proof of
  \(\vec{u}_j = \vec{v} \to \eps_x A(x; \vec{u}_j) = \eps_x A(x;
  \vec{u}_i)\).  It is trivial if \(i=j\), and otherwise we apply
  Lemma~\ref{l:subst-eps-=}.  We obtain \(\pi_e\), combining
  \(\vec{\pi}\) and \(\bar{\pi}\), which satisfies \(\cc(\pi_e) \le
  (w+1)\cdot k - 2 \cdot w\) and \(\mwidth(\pi_e, r) \le (w+1) \cdot
  \mwidth(\pi, r) - 2 \cdot w\).
\end{proof}

The following lemma concerns the case that the maximal critical
\eps-term is belonged to by both critical and \eps-equality formulas
at the same time.

\begin{lem}[Eliminating a maximal critical \eps-term]
  \label{l:critical-eps-term-elim}
  We deal with the case the maximal critical \eps-term is belonged to
  by both critical and \eps-equality formulas at the same time.
  Assume \(\pi\) is an \ECeEq-proof, \(e\) is the maximal \eps-term of
  \(\pi\), and \(g\) is the \eps-matrix of \(e\).  Let \(a\) be the
  arity of \(g\), \(p\) the property degree of \(e\), and \(r\) the
  rank of \(\pi\).  The critical count and the maximal width of
  \(\pi_e\) at \(r\) obtained in Lemma~\ref{l:ext-first-eps-eq-1} are
  bounded as follows: \(\mwidth(\pi_e, r) \le 2 \cdot (\mwidth(\pi,
  r))^2\) and \(\cc(\pi_e) \le (\cc(\pi) + a \cdot p)\cdot
  (\mwidth(\pi, r))^2\).
\end{lem}

\begin{proof}
  Assume \(e\) is of the form \(\eps_xA(x; \vec{u})\).  The proof
  \(\bar{\pi}\) of \((\bigwedge_{j<\widtheq_\pi(e)} \vec{v}_j \neq
  \vec{u}) \to E\) is due to \(\bar{\sigma}\) of
  \((\bigwedge_{i<\widtheps_\pi(e)} \neg A(t_i; \vec{u})) \to
  (\bigwedge_{j<\widtheq_\pi(e)} \vec{v}_j \neq \vec{u}) \to E\) and
  \(\sigma_i\) of \(A(t_i; \vec{u}) \to (\bigwedge_{j<\widtheq_\pi(e)}
  \vec{v}_j \neq \vec{u}) \to E\) for \(i < \widtheps_\pi(e)\).  By
  removing all critical and \eps-equality formulas belonging to \(e\)
  using ex falso quodlibet we obtain \(\bar{\sigma}\), hence
  \(\cc(\bar{\sigma}) \le \cc(\pi) - \width_\pi(e)\).  On the other
  hand, each critical formula belonging to \(e\) in \(\pi\) is gone
  due to \eps-elimination, \(\cc(\sigma_i) \le \cc(\pi) -
  \width_\pi(e)\).  Note that each \eps-equality formula belonging to
  \(e\) is removed due to the premise \(\bigwedge_{j<\widtheq_\pi(e)}
  \vec{v}_j \neq \vec{u}\).  Therefore,
  \begin{align*}
    \cc(\bar{\pi}) & \le \cc(\bar{\sigma}) +
    \textstyle\sum_{i<\widtheps_\pi(e)} \cc(\sigma_i) \\ & \le
    (\cc(\pi) - \width_\pi(e)) \cdot (\widtheps_\pi(e)+1).
  \end{align*}
  On the other hand concerning the width, the following condition holds.
  \begin{align*}
    \mwidth(\bar{\pi}, r) & \le \mwidth(\pi, r) \cdot (\widtheps_\pi(e)+1).
  \end{align*}
  The proof \(\pi_j\) of \(\vec{v}_j = \vec{u} \to E\) is due to
  \eps-elimination, replacing \(e\) by \(\eps_x A(x; \vec{v}_j)\).
  All \eps-equality formulas belonging to \(e\) is gone, each critical
  formula belonging to \(e\) is replaced by a critical formula
  belonging to \(\eps_x A(x; \vec{v}_j)\), and there are additional
  \eps-equality formulas due to Lemma~\ref{l:subst-CA}.  Therefore,
  \begin{align*}
    \cc(\pi_j) & \le \cc(\pi)- 1 + a \cdot p \cdot (\widtheps_\pi(e)+1)
  \end{align*}
  and
  \begin{align*}
    \mwidth(\pi_j, r) & \le \mwidth(\pi, r) + \width_{\pi}(e) - 1.
  \end{align*}
  Considering the construction of \(\pi_e\),
  \begin{align*}
    \cc(\pi_e) & \le \cc(\bar{\pi}) + \textstyle\sum_{j=0}^{\widtheq_\pi(e)}\cc(\pi_j) \\
    & = (\cc(\pi) - \width_\pi(e)) \cdot (\widtheps_\pi(e)+1) \\
    & \qquad \qquad + (\cc(\pi)- 1 + a \cdot p \cdot (\widtheps_\pi(e)+1)) \cdot \widtheq_\pi(e) \\
    & < \cc(\pi)\cdot(\widtheps_\pi(e)+\widtheq_\pi(e)+1)  + a \cdot p \cdot (\widtheps_\pi(e)+1)\cdot \widtheq_\pi(e) \\
    & < \cc(\pi) \cdot (\mwidth(\pi, r))^2 + a \cdot p \cdot (\mwidth(\pi, r))^2\\
    & = (\cc(\pi) + a \cdot p )\cdot (\mwidth(\pi, r))^2
  \end{align*}
  and
  \begin{align*}
    \mwidth(\pi_e, r) & \le \mwidth(\bar{\pi}, r) + \textstyle \sum_{j=0}^{\widtheq_\pi(e)} \mwidth(\pi_j, r) \\
    & = \mwidth(\pi, r) \cdot (\widtheps_\pi(e)+1) + (\mwidth(\pi, r) + \width_{\pi}(e) - 1) \cdot \widtheq_\pi(e) \\
    & < \mwidth(\pi, r)(\widtheps_\pi(e)+\widtheq_\pi(e)+1) + \width_\pi(e) \cdot \widtheq_\pi(e) \\
    & < 2 \cdot (\mwidth(\pi, r))^2
  \end{align*}
  Note that \(\width_\pi(e) \ge 2\), \(\width_\pi(e) >
  \widtheps_\pi(e)\), and \(\width_\pi(e) > \widtheq_\pi(e)\) hold.
  We conclude the claimed bounds due to
  Lemma~\ref{l:critical-eps-term-elim-ca} and
  Lemma~\ref{l:critical-eps-term-elim-eqeps}.
\end{proof}

The notation of \emph{hyperexponentiation} is useful to represent
large numerals.  We define the hyperexponentiation and give some
arithmetics on the exponentiation and the hyperexponentiation.

\begin{dfn}[Hyperexponentiation]
  For natural numbers \(k, n, m\), \(k_n^m\) is defined to be
  \begin{align*}
    & k_0^m := m, && k_{n+1}^m := k^{k_n^m}.
  \end{align*}
\end{dfn}

\begin{lem}
  For natural numbers \(x,y\), the following formulas hold.
  \begin{align*}
    & x+1 \le 2^x, && 2^x + y + 1 \le 2^x + 2^y, && x \cdot 2^y \le 2^{x+y}, && 2^x + 2^y \le 2^{x+y}+1, \\
    & (2_{x+2}^y)^2 \le 2_{x+2}^{y+1}, && x + 2^{2^x} \le 2^{2^{x+1}}, && x \cdot 2^{2^x} \le 2^{2^{x+1}}.
  \end{align*}
\end{lem}

\begin{lem}[Eliminating critical \eps-terms of maximal rank]
  \label{l:maximal-critical-eps-terms-elim}
  Assume \(\pi\) is an \EC-proof of \(E\) in
  Lemma~\ref{l:ext-first-eps-eq-4}.  Let \(r\), \(n\), \(w\), \(a\),
  and \(p\) be \(\rk(\pi)\), \(\order(\pi, r)\), \(\mwidth(\pi, r)\),
  the maximal arity \(\marity(\pi, r)\), and the maximal property
  degree \(\mpdgr(\pi, r)\), respectively.  The critical count of
  \(\pi'\) and the length of disjunction of \(E'\) obtained in
  Lemma~\ref{l:ext-first-eps-eq-4}, such that \(\rk(\pi')<\rk(\pi)\),
  are bounded as follows: \(\cc(\pi') \le (k + a \cdot p) \cdot
  2^{2^{(w+n)\cdot n}} \le (k + a \cdot p) \cdot 2_2^{3k^2}\) and
  \(\lth(E, E') \le 2^{2^{w+n+1}} \le 2_2^{3k+1}\) for \(k =
  \cc(\pi)\).
\end{lem}

\begin{proof}
  Consider a list of proofs \(\rho_0, \rho_1, \rho_2, \ldots,
  \rho_{n}\) where \(\rho_0 = \pi\), \(\rho_{n} = \pi'\), and each
  \(\rho_{j+1}\) is obtained by means of
  Lemma~\ref{l:critical-eps-term-elim} from \(\rho_j\), eliminating a
  critical \eps-term of rank \(r\).  We prove by induction on \(j\) that
  \(\mwidth(\rho_{j}, r) \le 2^{2^{j}-1} w^{2^j}\).  In case \(j=0\),
  \(\mwidth(\rho_{j}, r) = w\) and \(2^{2^0-1} w^{2^0}=w\).  In case \(j \mapsto
  j+1\),
  \begin{align*}
    \mwidth(\rho_{j+1}, r) & \le 2\cdot(\mwidth(\rho_{j}, r))^2 \\
    & \le 2 \cdot (2^{2^{j}-1} w^{2^j})^2 \qquad \text{by induction hypothesis} \\
    & = 2 \cdot 2^{2(2^j - 1)} w^{2 \cdot 2^j} \\
    & = 2^{2^{j+1}-1} \cdot w^{2^{j + 1}}.
  \end{align*}
  Also, \(2^{2^{j}-1} \cdot w^{2^{j }} \le 2^{2^{j}-1} \cdot 2^{w \cdot 2^{j}} \le 2^{(w +1)\cdot 2^{j}} \le 2^{w+j}_2\), therefore, \(\mwidth(\pi', r) \le 2_2^{w +j}\) holds.  Let \(k = \cc(\pi)\).  We also prove by
  induction on \(j\) that
  \begin{align*}
    \cc(\rho_{j+1}) & \le (\cc(\rho_j) + a \cdot p) \cdot (\mwidth(\rho_j, r))^2 \\
    & \le (\cc(\rho_j) + a \cdot p) \cdot (2^{2^{j}-1} \cdot w^{2^{j }})^2 = (\cc(\rho_j) + a \cdot p) \cdot 2^{2^{j+1}-2} \cdot w^{2^{j +1}}\\
    & \le ((k+a \cdot p) \cdot 2_2^{j(w+j)} + a \cdot p) \cdot 2^{2^{j+1}-2} \cdot w^{2^{j +1}} \qquad \text{by i.h.} \\
    & \le (k+a \cdot p) \cdot (2_2^{j(w+j)} + 1) \cdot 2^{2^{j+1}-2} \cdot w^{2^{j +1}} \\
    & \le (k+a \cdot p) \cdot 2^{2^{j(w+j)} + 1} \cdot 2^{2^{j+1}-2} \cdot 2^{2^{j +w}} \\
    & \le (k+a \cdot p) \cdot 2^{2^{j(w+j)} + 2^{j+1}-1 + 2^{j +w}} \\
    & \le (k+a \cdot p) \cdot 2^{2^{j(w+j) + j+w} + 2^{j+1}} \\
    & \le (k+a \cdot p) \cdot 2^{2^{j(w+j) + j+w + j+1}} = (k+a \cdot p) \cdot 2^{2^{(j+1)(w+j+1)}},
  \end{align*}
  hence \(\cc(\pi') \le (k+a \cdot p) \cdot 2^{2^{n(w+n)}}\).
  Since \(w = \mwidth(\pi) \le \cc(\pi)\) and \(n = \order(\pi,r)\le 2\cdot\cc(\pi)\),
  \(\cc(\pi') \le (k+a \cdot p) \cdot 2_2^{3k^2}\) holds.
  The bound of the length of disjunction \(\lth(E, E')\) is
  \(2^{2^{n+w+1}}\) because of the following calculation, hence \(\lth(E, E') \le 2^{2^{3k+1}}\).
  \begin{align*}
    \lth(E, E') & \le \textstyle \prod_{j=0}^{n} 2^{2^{j}-1} \cdot w^{2^{j }} \qquad \text{ due to the above bound for \(\mwidth(\rho_{n}, r)\)}\\
    & \le \textstyle \prod_{j=0}^{n} 2^{2^j - 1} \cdot \prod_{j=0}^{n} 2^{w \cdot 2^{j }} \\
    & = \textstyle 2^{\sum_{j=0}^n (2^j - 1)} \cdot 2^{\sum_{j=0}^n w \cdot 2^{j }}\\
    & = 2^{2^{n+1}-1 - n -1} \cdot 2^{w \cdot (2^{n+1 }-1)}\\
    & < 2^{(w+1) \cdot 2^{n+1}} \le 2^{2^w \cdot 2^{n+1}} = 2^{2^{n+w+1}}.
  \end{align*}
\end{proof}

\begin{thm}[Extended first epsilon theorem]
  \label{t:extended-first-epsilon-theorem-complexity}
  Assume \(\pi\) is an \ECeEq-proof of \(E(\vec{s}\,)\), an
  \ECeEq-formula where terms other than \(\vec{s}\) are \eps-free.
  There is an \ECEq-proof \(\pi'\) of the \EC-formula \(\bigvee_{i<n}
  E(\vec{t}_i)\) for some \eps-free terms \(\vec{t}_0, \ldots,
  \vec{t}_{n}\) such that \(n \le 2_{2k}^{6k^2+2k+a\cdot p}\), where
  \(k\) is the critical count \(\cc(\pi)\), \(a\) is the maximal arity
  \(\marity(\pi)\), and \(p\) is the maximal property degree
  \(\mpdgr(\pi)\).
\end{thm}

\begin{proof}
  We make a sequence of proofs \(\pi_0, \pi_1, \ldots, \pi_m\) where
  \(\pi_0\) is obtained by applying Lemma~\ref{l:ext-first-eps-eq-5}
  to \(\pi\), \(m\) is \(|\crs(\pi)|\), and \(\pi_{i+1}\) is obtained
  by applying Lemma~\ref{l:maximal-critical-eps-terms-elim} and
  Lemma~\ref{l:ext-first-eps-eq-5} in this order to \(\pi_i\).  We
  prove that \(\cc(\rho_{i}) \le 2_{2i}^{6k^2+a \cdot p + 2i}\)
  by induction on \(i\).  It is trivial in case \(i=0\).  In case \(i
  \mapsto i+1\),
  \begin{align*}
    \cc(\rho_{i+1}) & \le (\cc(\rho_{i}) + a \cdot p) \cdot 2_2^{{6 \cdot \cc(\rho_{i})}^2} \\
    & \le (2_{2i}^{6k^2+a \cdot p + 2i} + a \cdot p) \cdot 2_2^{6 \cdot (2_{2i}^{6k^2+a \cdot p + 2i})^2} \\
    & \le 2_{2i}^{6k^2+a \cdot p + 2i} \cdot 2_2^{6 \cdot (2_{2i}^{6k^2+a \cdot p + 2i})^2} + a \cdot p \cdot 2_2^{6 \cdot (2_{2i}^{6k^2+a \cdot p + 2i})^2} \\
    & \le 2_2^{6 \cdot (2_{2i}^{6k^2+a \cdot p + 2i})^2 + 1} + a \cdot p \cdot 2_2^{6 \cdot (2_{2i}^{6k^2+a \cdot p + 2i})^2} \\
    & \le 2_{2(i+1)}^{6k^2+a \cdot p + 2i+ 1} + 2_{2(i+1)}^{6k^2+a \cdot p + 2i+1} \\
    & \le 2_{2(i+1)}^{6k^2+a \cdot p + 2(i+ 1)}.
  \end{align*}
  We calculate the length \(\lth(E, E_{m}) = \lth(E, E_1) \cdot
  \lth(E_1, E_2) \cdot \cdots \cdot \lth(E_{m-1}, E_{m})\).
  \begin{align*}
    \lth(E, E_{m}) & \le 2^{2^{3k+1}} \cdot 2^{2^{3(2_{2}^{6k^2+a \cdot p + 2})+1}} \cdot \cdots \cdot 2^{2^{3(2_{2(m-1)}^{6k^2+a \cdot p + 2(m-1)})+1}} \\
    & < \Big(2^{2^{3\big(2_{2(m-1)}^{6k^2+a \cdot p + 2(m-1)}\big)+1}}\Big)^{m-1} \\
    & \le 2^{(m-1)\cdot 2^{3\big(2_{2(m-1)}^{6k^2+a \cdot p + 2(m-1)}\big)+1}}\\
    & < 2_{2m}^{6k^2+a \cdot p + 3m}
  \end{align*}
  Since \(|\crs(\pi)| \le \cc(\pi) = k\), \(n \le
  2_{2k}^{6k^2+3k+a\cdot p}\).
\end{proof}

\subsection{Alternative \eps-equality formula and closure}
\label{ss:alt}

We study \ECeEqOne, the system \ECe with \EQ and the following
\eps-equality formula,
\begin{align}
  \label{f:eps-equality-1}
  u_i=v \to \eps_x A(x; \vec{u}) = \eps_x A(x; \vec{u}_{i \mapsto v})
\end{align}
where \(\vec{u}_{i \mapsto v}\) is \(u_1, u_2, \ldots, u_{i-1}, v,
u_{i+1}, \ldots, u_n\); obtained by replacing the \(i\)-th element of
\(\vec{u}\) by \(v\).  We say that \(i\) is the \emph{position} of the
\eps-equality formula, and call the above formula the
\emph{\eps-equality formula of position \(i\)}.  Assuming \(\eps_x
A(x, a)\) is an \eps-matrix, Hilbert and Bernays employed the
following \eps-equality formula
\begin{align*}
  u=v \to \eps_x A(x, u) = \eps_x A(x, v),
\end{align*}
and presented the \eps-elimination method~\cite{HilbertBernays39}.
The formula (\ref{f:eps-equality-1}) in \ECeEqOne explicitly expresses
the notion of position.  An upper bound of the Herbrand complexity for
\ECeEqOne can be independent from the maximal arity of critical
\eps-matrices in the given proof, in contrast to the case for \ECeEq.
The following lemma tells us that \ECeEqOne is as strong as \ECeEq.
\begin{lem}
  \(\ECeEq \vdash A\) if and only if \(\ECeEqOne \vdash A\).
\end{lem}
\begin{proof}
  Assume \(\pi\) is an \ECeEq-proof of A.  We prove \(\vec{u}=\vec{v}
  \to \eps_x A(x; \vec{u})=\eps_x A(x; \vec{v})\) in \ECeEqOne by
  induction on the number of differences between \(\vec{u}\) and
  \(\vec{v}\).  If there is one difference, say \(u_i\) and \(v_i\),
  the formula is an instance of the alternative \eps-equality formula
  \(u_i = v_i \to \eps_x A(x; \vec{u})=\eps_x A(x; \vec{v})\).  Assume
  \(\pi\) is an \ECeEqOne-proof of the formula \(\vec{u}=\vec{v}
  \to \eps_x A(x; \vec{u})=\eps_x A(x; \vec{v})\) where there are \(n\)
  differences between \(\vec{u}, \vec{v}\) and \(u_k = v_k\).  For any
  \(w\), \(v_k = w \to \eps_x A(x; \vec{v}) = \eps_x A(x; \vec{v}_{k
    \mapsto w})\) is an instance of the alternative \eps-equality
  formula, hence by the transitivity of equality and the deduction
  theorem, \(\vec{u}=\vec{v}_{k \mapsto w} \to \eps_x A(x; \vec{u}) =
  \eps_x A(x; \vec{v}_{k \mapsto w})\).  The other direction is trivial.
\end{proof}

We define the \emph{closure} \(\Closure(\pi,g)\) of the \eps-matrix
\(g\) in the proof \(\pi\), which is to enumerate all the critical
\eps-terms of \(g\) which may occur during the \eps-elimination
process.

\begin{dfn}[Closure]
  Let \(\pi\) be an \ECeEqOne-proof and \(g\) the maximal \eps-matrix
  of \(\pi\).  We define a set \(\mathcal{T}_\pi (g,i)\) of terms
  which occur at the premise of the \eps-equality formulas of the
  position \(i\) in \(\pi\), and a strict order \(\prec_{g,i}\) such
  that for \(u, v \in \mathcal{T}_\pi (g,i)\), \(u \prec_{g,i} v\) iff
  \(\dgr(u) \le \dgr(v) ~\text{and}~ u \not\aeq v\).  We define a
  partially ordered set \(\Closure(\pi,g)\) to be \(\{g(\vec{u}) \mid
  u_i \in \mathcal{T}_\pi (g,i)\}\) with the transitive closure of the
  strict order \(\prec_g\) such that for any position \(i\),
  \(g(\vec{u}) \prec_g g(\vec{u}_{i \mapsto v})\) iff \(u_i
  \prec_{g,i} v\).
\end{dfn}

It is easy to see that a closure is a lattice.

\begin{dfn}[Strict partial order on closures]
  Let \(M, N\) be sublattices of \(\Closure(\pi, g)\).  Define a
  strict partial order as follows: \(M \prec N\) if and only if the
  upper bounds of \(N\) is a proper subset of the upper bounds of
  \(M\).
\end{dfn}

Due to the above strict partial order, we can choose terms among maximal
\eps-terms such that they are also maximal due to \(\prec\).

\begin{dfn}[\(\prec\)-maximal critical \eps-term]
  Assume \(\pi\) is a proof and \(e\) is a critical \eps-term of
  \eps-matrix \(g\) in \(\pi\).  If \(e\) is maximal and \(e' \prec
  e\) holds for all the other critical \eps-term \(e'\) of \(g\) in
  \(\pi\), \(e\) is a \(\prec\)-maximal critical \eps-term n \(\pi\).
\end{dfn}

In the rest of this section, we simply say maximality to mean the
\(\prec\)-maximality.  We introduce an order on proofs.

\begin{dfn}[Strict partial order on proofs due to a closure]
  Let \(\pi\) and \(\pi'\) be proofs.  We define \(\pi \prec \pi'\) if
  and only if either \(\rk(\pi) < \rk(\pi')\) or \(\rk(\pi) =
  \rk(\pi')\) and \(M \prec_g M'\) for all \(g\) of the rank
  \(\rk(\pi)\), where \(M\) and \(M'\) are given by the sets of
  critical \eps-terms of \(g\) in \(\pi\) and \(\pi'\), respectively.
\end{dfn}

Instead of Lemma~\ref{l:aux-1}, we use the following identity lemma in
our current setting.

\begin{lem} \label{l:identity-var}
  Assume \(A(a; \vec{b})\) has exactly one occurrence of each \(b_i\)
  and for each \(b_i\), if it is a subterm of some term \(t\), \(a \in
  \NQVar(t)\) holds.  For any term \(s\) there exists \(\pi\) and
  \(\ECeEqOne \vdash_\pi u_i=v \to A(s; \vec{u}) \to A(s; \vec{u}_{i
    \mapsto v})\), such that \(\cc(\pi) \le \dgr(A(a; \vec{b}))\) and
  \(\rk(\pi) \le \rk(A(a; \vec{b}))\).  Moreover, concerning the
  \eps-equality formulas \(B_1, \ldots, B_n\) used in \(\pi\), if
  \(B_i\) belongs to \(e_i\) and \(e_i'\), \(e_i\) and \(e_i'\) are
  proper subterms of \(e_{i+1}\) and \(e_{i+1}'\), \(e_i\) and
  \(e_i'\) occur just once in \(e_{i+1}\) and \(e_{i+1}'\),
  respectively, \(\dgr(e_{i+1})=\dgr(e_{i}) + 1\) and
  \(\dgr(e_{i+1}')=\dgr(e_{i}') + 1\).
\end{lem}

\begin{proof}
  Trivial due to the proof of Lemma~\ref{l:aux-1}.
\end{proof}

If no \eps-equality formula belongs to the maximal critical \eps-term
to eliminate, we apply Lemma~\ref{l:critical-eps-term-elim-ca}.
Otherwise the following lemma applies.

\begin{lem}[Eliminating a maximal critical \eps-term involving \eps-equality]
  \label{l:eps-eq-elimination-closure}
  Let \(\pi\) be an \ECeEqOne-proof of \(E(\vec{e})\) where
  \(E(\vec{a})\) is an \eps-free formula in \ECeEqOne and \(\vec{e}\)
  is a list of \eps-terms.  Assume the maximal critical \eps-term of
  \(\pi\) is \(\eps_x A(x; \vec{u}) =: e\) and in \(\pi\) there are
  critical formulas including ones of the form \(A(t_l, \vec{u}) \to
  A(\eps_x A(x; \vec{u}), \vec{u})\) for \(1 \le l \le m\) as well as
  \eps-equality formulas as follows.
  \begin{align*}
    & u_{i_1} = v_1 \to \eps_x A(x; \vec{u}) = \eps_x A(x; \vec{u}_{i_1 \mapsto v_1}) \\
    & \vdots \\
    & u_{i_n} = v_n \to \eps_x A(x; \vec{u}) = \eps_x A(x; \vec{u}_{i_n \mapsto v_n})
  \end{align*}
  There is a proof \(\pi_e\) of \(\bigvee_{i \le \width(\pi,e)}
  E(\vec{r}_i)\) for some \(\vec{r}_i\) where there is no critical
  occurrence of \(\eps_x A(x; \vec{u})\) in \(\pi_e\), \(\pi_e \prec
  \pi\), \(\cc(\pi_e) \le 2\cdot (p +1) \cdot \cc(\pi)^2\),
  \(\mwidth(\pi_e, r) \le 2 \cdot \mwidth(\pi, r)^2\), and
  \(|\Closure(\pi_e,g)|\le|\Closure(\pi,g)|\) for any critical
  \eps-matrix \(g\) of rank \(r\) in \(\pi_e\), where \(p :=
  \pdgr(e)\), \(r := \rk(\pi)\) and \(w := \width(\pi, e)\).
\end{lem}

\begin{proof}
  We give proofs \(\rho_j\) of \(u_{i_j}= v_j \to E(\vec{s}\{e/\eps_x
  A(x; \vec{u}_{i_j \mapsto v_j})\})\) for \(1 \le j \le n\) and also a
  proof \(\sigma\) of \((\bigwedge_{1 \le j \le n} \neg u_{i_j}= v_j)
  \to E(\vec{e})\), such that they are all free from the critical
  \eps-term \(e\).  For the former ones, we apply the substitution
  \(\{e/\eps_x A(x; \vec{u}_{i_j \mapsto v_j})\}\) throughout the proof
  \(\pi\).  Then, the critical formulas in \(\pi\) belonging to \(e\)
  goes to \(A(t_l, \vec{u}) \to A(\eps_x A(x; \vec{u}_{i_j \mapsto
    v_j}), \vec{u})\), which is provable by the following three steps
  \begin{align*}
    & A(t_l, \vec{u}) \to A(t_l, \vec{u}_{i_j \mapsto v_j}) \\
    & A(t_l, \vec{u}_{i_j \mapsto v_j}) \to A(\eps_x A(x; \vec{u}_{i_j
      \mapsto v_j}), \vec{u}_{i_j \mapsto v_j}) \\
    & A(\eps_x A(x; \vec{u}_{i_j
      \mapsto v_j}), \vec{u}_{i_j \mapsto v_j}) \to A(\eps_x A(x; \vec{u}_{i_j
      \mapsto v_j}), \vec{u})
  \end{align*}
  where the first and the third formulas by
  Lemma~\ref{l:identity-var}, and the second one is a critical
  formula.  On the other hand, the \(k\)-th \eps-equality formulas in
  \(\pi\) belonging to \(e\) goes to
  \begin{align*}
    u_{i_k} = v_k \to \eps_x A(x; \vec{u}_{i_j \mapsto v_j}) = \eps_x
    A(x; \vec{u}_{i_k \mapsto v_k}),
  \end{align*}
  which is provable as follows.  If the positions of \(j\)-th
  and \(k\)-th \eps-equality formulas are the same,
  namely, \(i_j=i_k\), the assumption \(u_{i_j} = v_j\) and the
  premise \(u_{i_k} = v_k\) implies \(v_j = v_k\), hence the above
  formula is proven by means of an \eps-equality formula \(v_j = v_k
  \to \eps_x A(x; \vec{u}_{i_j \mapsto v_j}) = \eps_x A(x; \vec{u}_{i_k
    \mapsto v_k})\).  Otherwise, we prove the above formula by using
  the two \eps-equality formulas.
  \begin{align*}
    & u_{i_k} = v_k
    \to \eps_x A(x; \vec{u}_{i_j \mapsto v_j}) = \eps_x A(x; \vec{u}_{i_j \mapsto v_j, i_k
      \mapsto v_k}) \\
    & v_j = u_{i_j}
    \to \eps_x A(x; \vec{u}_{i_j \mapsto v_j,  i_k
      \mapsto v_k}) = \eps_x A(x; \vec{u}_{i_k
      \mapsto v_k})
  \end{align*}
  In any case, the critical \eps-term \(e\) is gone.  In the second
  case, although the critical \eps-term \(\eps_x A(x; \vec{u}_{i_j
    \mapsto v_j, i_k \mapsto v_k})\) may be new in the sense that it
  has no critical occurrence in \(\pi\), the obtained proof is
  strictly smaller than the \(\pi\) with respect to \(\prec\) in
  \(\Closure(\pi, \eps_x A(x; \vec{a}))\).  On
  the other hand for the proof \(\sigma\), these \eps-equality
  formulas are provable due to ex falso quodlibet, hence we use
  Lemma~\ref{l:critical-eps-term-elim-ca} to get rid of the critical
  formulas.

  Finally, we discuss the bounds.  After the elimination process, each
  \(\rho_j\) may have an increased number for \(\mwidth(\rho_j, r)\),
  which is bounded by \(2 \cdot \mwidth(\pi, r)\).  Due to
  Lemma~\ref{l:critical-eps-term-elim-ca}, \(\mwidth(\sigma, r) \le
  \mwidth(\pi, r) \cdot (m+1)\).  Therefore, \(\mwidth(\pi_e, r) \le 2
  \cdot \mwidth(\pi, r) \cdot n +\mwidth(\pi, r) \cdot (m+1) \le 2
  \cdot \mwidth(\pi, r)^2\).  The critical counts for \(\vec{\rho},
  \sigma\) are bounded as \(\cc(\rho_j) \le \cc(\pi) + 2 \cdot (p
  \cdot m + n - 1)\) and \(\cc(\sigma) \le (\cc(\pi)-n) \cdot (m+1)\),
  hence \(\cc(\pi_e) \le n\cdot(\cc(\pi) + 2 \cdot (p \cdot m + n -
  1)) + (\cc(\pi)-n) \cdot (m+1) \le (\cc(\pi) + 2\cdot \mwidth(\pi,
  r)\cdot p) \cdot (\mwidth(\pi, r) + 1) \le 2 \cdot (p+1) \cdot
  \cc(\pi)^2\).
\end{proof}

In case we deal with \eps-equality formulas with different positions,
we may introduce a critical \eps-term which is not a critical
\eps-term in the original proof \(\pi\).  The closure covers all the
potential new critical \eps-terms beforehand.  Also note that the
closure won't be expanded through the \eps-elimination procedure,
because the maximal critical \eps-term never matches terms in the
premise of \eps-equality formulas, hence we don't get anything new for
\(\mathcal{T}_\pi (g,i)\).

By Lemma~\ref{l:eps-eq-elimination-closure}, we get a proof without
using a maximal \eps-term in the corresponding closure, hence the size
of the closure is an upper bound for the number of applications of the
lemma needed to eliminate the \eps-matrix.  The size of the closure
\(|\Closure(\pi,g)|\) is bounded by \(\prod_{i \in I_g}
|\mathcal{T}_\pi (g,i)|\), where \(I_g\) is the set of positions of
\eps-equality formulas belonging to \(g\) in \(\pi\).  In the
following proof the upper bound of this size is estimated only by the
number of \eps-equality formulas.

\begin{lem}\label{l:closure-size-bounded}
  For any proof \(\pi\) and \eps-matrix \(g\) in \(\pi\), \(|\Closure(\pi,g)| \le 2^{\widtheq(\pi, g)^2}\).
\end{lem}
\begin{proof}
The number of different positions of \eps-equality formulas belonging
to \(g\) is at most \(\widtheq(\pi, g)\), and \(|\mathcal{T}_\pi(g,i)|
\le w+1\) holds for each \(i\), hence \(|\Closure(\pi,g)|\) is bounded
by \((\widtheq(\pi, g)+1)^{\widtheq(\pi, g)}\) which is smaller than
\(2^{\widtheq(\pi, g)^2}\).
\end{proof}

Repeatedly eliminating a maximal critical \eps-term,
the rank is diminished.

\begin{lem}[Eliminating critical \eps-terms of the maximal rank]
  \label{l:eps-eq-elimination-rank}
  Let \(\pi\) be an \ECeEqOne-proof of \(E(\vec{e}\,)\) where
  \(E(\vec{a})\) is an \eps-free formula in \ECeEqOne and \(\vec{e}\)
  is a list of \eps-terms.  There is an \ECeEqOne-proof \(\pi'\) of
  \(\bigvee_{i \le N} E(\vec{r}_i)\) for some quantifier free terms
  \(\vec{r}_i\) and \(N\) bounded by \(2_3^{2\cc(\pi)^2+3\cc(\pi)}\) such that
  \(\rk(\pi') < \rk(\pi)\) and \(\cc(\pi') \le 2_3^{(\cc(\pi) + 1)^2 + p}\),
  where \(p\) is \(\pdgr(\pi)\).
\end{lem}

\begin{proof}
  We repeatedly apply Lemma~\ref{l:eps-eq-elimination-closure} in
  order to eliminate maximal critical \eps-terms.  As
  Lemma~\ref{l:eps-eq-elimination-closure} eliminates an \eps-term,
  which is a least upper bound of a closure, without expanding the
  closures of rank \(r\), this process is terminating.  Let \(g_1,
  g_2, \ldots, g_{\matrices(\pi,r)}\) be the critical \eps-matrices of
  rank \(r\) in \(\pi\), then the number of applications of the lemma
  is bounded by the number of critical formulas of rank \(r\) plus the
  number of elements in the closures \(\Closure(\pi, g_i)\), namely,
  by \(\cc(\pi) + \sum_{i \in \matrices(\pi, r)}|\Closure(\pi, g_i)|
  \le 2^{\cc(\pi)\cdot(\cc(\pi)+1)}\).  Let \(k\) be \(\cc(\pi)\).
  Solving the recurrence relations \(w_0 = \mwidth(\pi, r)\) and
  \(w_{n+1} = 2\cdot w_n^2\), we find \(w_n = 2^{(2^n - 1)} \cdot
  (w_0)^{2^n}\), hence \(w_{2^{k(k+1)}} = 2^{(2_2^{k(k+1)} - 1)} \cdot
  (w_0)^{2_2^{k(k+1)}} < 2_3^{2k^2+3k}\) is a bound for the length
  \(N\) of the Herbrand disjunction.  Also solving the recurrence
  relations \(k_0 = k\) and \(k_{n+1} = 2\cdot (p +1)\cdot k_n^2\),
  \(k_n = (2 \cdot (p +1))^{2^n-1} \cdot k^{2^n} \le 2_2^{k+p+n}\),
  hence \(\cc(\pi') \le 2_2^{k+p+2^{k(k+1)}} \le 2_3^{(k + 1)^2 + p}\)
  by excluding a trivial case \(k=0\).
\end{proof}

\begin{thm}[Extended first epsilon theorem for \ECeEqOne]
  Assume \(\ECeEqOne \vdash_\pi E(\vec{s}\,)\) for \(E(\vec{a}) \in
  L(\ECEq)\) and \(\vec{s} \in L(\ECeEq)\), then \(\ECEq \vdash
  \bigvee_{i=0}^{N} E(\vec{s}_i)\) for some \(\vec{s}_{i} \in
  L(\EC)\), where \(N\) is bounded by \(2_{3k}^{k^2 + 3 k + p + 2} \le
  2_{3k+2}^{k + p}\) for \(p := \mpdgr(\pi)\).
\end{thm}

\begin{proof}
  As long as the maximal rank of critical formulas in \(\pi\) is
  smaller than \(\rk(\pi)\), we repeatedly apply
  Lemma~\ref{l:ext-first-eps-only-eq}.  Our proof proceeds by
  induction on the number of different ranks of critical formulas in
  \(\pi\), namely, size of \(S:=\{\rk(e) \mid e \text{ belongs to a
    critical formula in } \pi\}\).  In each step we use
  Lemma~\ref{l:eps-eq-elimination-closure}, and then
  Lemma~\ref{l:ext-first-eps-only-eq} in the same way as mentioned
  above.  Nullary constants can substitute \eps-terms which still
  remained.  The number of applications of
  Lemma~\ref{l:eps-eq-elimination-rank} is \(|S|\) which is bounded by
  \(\cc(\pi)\).  Solving the recurrence relations \(k_0 = k\) and
  \(k_{n+1} = 2_2^{2_3^{k_n} + 2^p}\), we find \(k_n \le 2_{3n}^{k^2 +
    2k+p+n}\).  The length \(N\) of the Herbrand disjunction is
  bounded by \(2_{3k}^{k^2 + 3 k + p + 2}\), taking \(\cc(\pi)=k\) for
  the bound of \(n\), because \(\prod_{i<n}w_{i} \le (w_{n-1})^n =
  (2_3^{2(k_{n-1})^2 + 3 k_{n-1}})^n \le 2_{3n}^{k^2 + 2 k + p + n +
    2}\).
\end{proof}

The property degree \(p\) is independent from the critical count, and
the use of \(p\) in the upper bound of the length of Herbrand
disjunction is an explanation of potential complexity due to the
presence of the \eps-equality formula.

\section{Lower Bounds on Herbrand Disjunctions}
\label{s:lower-bounds}

In this section, we adapt the result by Statman to give the lower
bounds on Herbrand disjunctions~\cite{Statman79}.  The case without
equality was studied by Moser and Zach, where Orevkov's result was
used to give the lower bound of Herbrand
complexity~\cite{Orevkov82,MoserZach06}.  We first define a
combinatory logic, which allows to define a term whose normal form is
hyperexponentially long.  Then we can describe a sequence of formulas
such as the critical counts of their proofs shows a linear growth in
\PCEq.  Finally, we show that the Herbrand complexity of those
formulas are hyperexponential.

\begin{dfn}
  Let \(\circ\) be a left associative binary function symbol.
  The combinators \(S\), \(B\), \(C\), and \(I\) are defined as
  nullary constants satisfying
  \begin{align*}
    & S \circ x \circ y \circ z = x \circ z \circ (y \circ z), && B \circ x \circ y \circ z = x\circ (y\circ z) \\
    & C \circ x \circ y \circ z = x \circ z \circ y, && I \circ x = x.
  \end{align*}
  Let \lamI denote the set of the above semi-formulas and \(\alllamI\)
  the set of universally quantified closed formulas of \lamI.
\end{dfn}
From now, we omit \(\circ\) and write terms in a common way, e.g.~\(S
x y z\) and etc.
\begin{dfn}
  Define \(T := S B (C B I)\) and \(T_1 := T\), \(T_{n+1} := T_n T\).
\end{dfn}
For a finite set of formulas \(S := \{A_1, A_2, \ldots, A_n\}\), let
\(S \to B\) denote \(A_1 \to A_2 \to \ldots \to A_n \to B\).
\begin{lem}
  \label{l:eq-T}
  There is a \PCEq-proof \(\pi\) such that \(\vdash_\pi \alllamI \to T t u = t (t u)\).
\end{lem}
\begin{proof}
  The proof \(\pi\) involves the equality axioms, the combinatory
  logic axioms, and the quantifier axiom, so that the following
  equations are proven.
  \begin{align*}
    (S B) (C B I) t u & = (B t) (C B I t) u = t (C B I t u) = t (B t I u) = t (t (I u)) = t (t u).
  \end{align*}
  Easy to find a \(\pi\) whose critical count \(\cc(\pi)\) is 12,
  which is constant.
\end{proof}
\begin{dfn}
Let \(p\) and \(q\) be nullary constants.  For any \(n\ge 1\) we
define comprehension terms \(H_{n}\) and a formula \(E_n\).
\begin{align*}
      & H_1 := \{y \mid \forall z. pz = p(yz)\}, && H_{n+1} := \{y
  \mid \forall z \in H_n. y z \in H_n\}, \\
  & E_n := p q = p(T_n q q).
\end{align*}
\end{dfn}
\begin{lem}
  \label{l:lb-0}
  In \(\PCEq\), \(\vdash \alllamI \to z(z y)\vec{t} \in H_{m}\) implies
  \(\vdash \alllamI \to T z y \vec{t} \in H_{m}\) for any terms
  \(\vec{t}\) of arbitrary length and for any \(m\ge 1\).
\end{lem}
\begin{proof}
  By induction on \(m\).  In the base case, we assume that there
  exists a proof of \(z (z y) \vec{t} \in H_1\), that is, \(\forall
  x. p x = p(z (z y) \vec{t} x)\), and prove \(T z y \vec{t} \in
  H_1\), namely, \(\forall x. p x = p(T z y \vec{t} x)\).  We apply
  Lemma~\ref{l:eq-T} once and equality axioms to make a proof.  In the
  step case, we prove that if \(\vdash z (z y) \vec{t} \in H_{m+1}\)
  then \(\vdash T z y \vec{t} \in H_{m+1}\) for any \(\vec{t}\) of any
  length.  Assume \(\vec{t}\) and \(\vdash z (z y) \vec{t} \in
  H_{m+1}\), and we prove \(\vdash T z y \vec{t} \in H_{m+1}\), which
  unfolds into \(\forall x \in H_m. T z y \vec{t} x \in H_m\).
  Further assume \(x \in H_m\), then the induction hypothesis,
  \(\vdash z (z y) \vec{t} \in H_{m}\) implies \(\vdash T z y \vec{t}
  \in H_{m}\) for any \(\vec{t}\) of any length, is applicable with
  terms \(\vec{t},x\), hence it suffices to show \(\vdash z (z y)
  \vec{t} x \in H_m\) which is trivial due to the assumptions \(\vdash
  z (z y) \vec{t} \in H_{m+1}\), which unfolds into \(\vdash \forall x
  \in H_m. z (z y) \vec{t} x \in H_{m}\), and \(x \in H_m\).
\end{proof}

\begin{lem}
  \label{l:lb-1}
  In \(\PCEq\), \(\vdash \alllamI \to T \in H_{m+1}\) for any \(m \ge 1\).
\end{lem}
\begin{proof}
  In case \(m=1\), \(T \in H_2\) holds trivially.
  Assume \(m>1\).  Unfolding \(T \in H_{m+2}\).
  \begin{align*}
    T \in H_{m+2} & \Longleftrightarrow \forall z \in H_{m+1}. T z \in H_{m+1} \\
    & \Longleftrightarrow \forall z. (\forall y \in H_m. z y \in H_{m}) \to \forall y \in H_m. T z y \in H_{m}.
  \end{align*}
  Due to its premise, \(z(z y) \in H_m\) follows.  Lemma~\ref{l:lb-0}
  implies \(T z y \in H_m\).
\end{proof}

\begin{lem}
  \label{l:lb-2}
  There is a \PCEq-proof \(\pi\) such that \(\vdash_\pi \alllamI \to
  T_{n} \in H_{2}\) for any \(n \ge 1\), such that \(\cc(\pi)\) is
  linear in \(n\).
\end{lem}
\begin{proof}
  We prove that in \(\PCEq\), \(\vdash \alllamI \to T_{n} \in H_{m+1}\) for any
  \(n, m \ge 1\) by induction on \(n\).  The base case, \(T_1 \in
  H_{m+1}\) for any \(m\), is proven by Lemma~\ref{l:lb-1}.  In order
  to prove the step case, \(T_{n+1} \in H_{m+1}\) for any \(m\), it
  suffices to assume \(m\) and to prove both \(T_n \in H_{m+2}\) and
  \(T \in H_{m+1}\).  They are straightforward by induction hypothesis
  and by Lemma~\ref{l:lb-1}, respectively.
\end{proof}
Assuming \alllamI and \(\forall x. p x = p(q x)\), there is a linear
proof of \(E_n\) in \PCEq.
\begin{thm}
  \label{t:linear-proof}
  There is a \PCEq-proof \(\pi_n\) such that \(\vdash_{\pi_n}
  \alllamI \to (\forall x. p x = p(q x)) \to E_n\) and \(\cc(\pi_n)\)
  is linear in \(n\).
\end{thm}
\begin{proof}
  Assume \alllamI and \(\forall x. p x = p(q x)\). By
  Lemma~\ref{l:lb-2}, there is a proof \(\pi\) of \(T_n \in H_2\) with
  \(\cc(\pi)\) being linear in \(n\).  The formula unfolds into
  \(\forall y. (\forall z. p z = p (y z)) \to \forall z. p z = p(T_n y
  z)\),
  hence \(p q = p(T_n q q)\).
\end{proof}

For a set \(X\) of semi formulas, let \(X^*\) denote a set of closed
formulas obtained by instantiating each formula in \(X\).  We use the
following lemma, whose proof is found in Statman~\cite{Statman79}, to
prove the main theorem.

\begin{lem}[Statman~\cite{Statman79}]
  Suppose that \(X\) is a finite subset of \(\{p x = p(q x)\}^*\) such
  that \(\vdash \lamIstar \to X \to E_n\); then there is a finite
  subset \(Y\) of \(\{p x = p(q x)\}^*\) such that \(\vdash \lamIstar
  \to X \to E_n\), \(|Y| \le |X|\), and each term occurring in \(Y\)
  is closed and in normal form.
\end{lem}
\begin{thm}[Statman~\cite{Statman79}]
  \label{t:Statman}
  Suppose \(X\) is a finite subset of \(\{p x = p(q x)\}^*\) such that
  \(\vdash \lamIstar \to X \to E_n\) and each term occurring in \(X\) is
  in normal form; then \(|X| \ge 2_n^1 / 2\).
\end{thm}
\begin{proof}
  Suppose to the contrary that the size of \(X\) is less than \(2_n^1
  / 2\).  Assume \(X = \{p M_i = p(q M_i) \mid 1 \le i \le m\}\) for
  some \(m < 2_n^1 / 2\) and \(M_i\).  Because the number of possible
  different instantiations is at most \(m\), for some \(k\) with \(1 <
  k \le 2_n^1 +1\), for any \(i\), \(1 \le i \le m\), neither \(M_i\)
  nor \(q M_i\) is same as \(q^k\).  Let \(0, 1\) be new constants.
  We define \alllamIplus to be \alllamI extended with infinitely many
  equations, so that the following reduction rules for closed \(M\) of
  normal form without \(p\) in function position are available:
  \begin{align*}
    & p M \rhd 0 \qquad \text{if \(M = q^j\) for \(j < k\),} \\
    & p M \rhd 1 \qquad \text{otherwise.}
  \end{align*}
  We prove that \(\not \vdash \alllamIplus \to E_n\).  The normal form
  of formula \(p q = p q^{2_n^1 + 1}\) is \(0 = 1\), which is not
  provable in \alllamIplus because of the Church-Rosser property,
  cf.~Theorem 3 in~\cite{Hindley74}.  Notice that \(p
  q^{2_n^1 + 1}\) is the normal form of \(p (T_n q q)\).

  Finally, we prove that \(\vdash \alllamIplus \to p M_i = p (q M_i)\)
  for each \(i\).  If \(M_i\) does not contain \(p\) in function
  position, we consider the following three cases.  If \(M_i\) is not
  of the form \(q^j\) for any \(j\), \(p M_i, p (q M_i) \rhd 1\), and
  if \(M_i\) is \(q^j\) and \(j, j+1 < k\), \(p M_i, p (q M_i) \rhd
  0\), otherwise \(k< j, j+1\) and then \(p M_i, p (q M_i) \rhd 1\).
  If \(M_i\) contains \(p\) in function position, \(M_i\) has a normal
  form containing \(0\) or \(1\) and without \(p\) in function
  position, so \(p M_i, p (q M_i) \rhd 1\) under \alllamIplus.

  As \alllamIplus extends \lamIstar, \(\not
  \vdash \alllamIplus \to X \to E_n\) implies \(\not
  \vdash \lamIstar \to X \to E_n\).
\end{proof}
This theorem implies the following corollary which tells us that the
length of Herbrand disjunction of the formula in
Theorem~\ref{t:linear-proof} is hyperexponential.  For a set of semi
formulas \(\{A_0(\vec{x}_0), \ldots, A_n(\vec{x}_n)\} =: X\), let
\(X(\vec{t}_0, \ldots, \vec{t}_n\,)\) denote a conjunction
\(A_0(\vec{t}_0) \land \ldots \land A_n(\vec{t}_n)\) .
\begin{cor}
  \(\HC(\alllamI \to (\forall x. p x = p(q x)) \to E_n) \ge 2_n^1/2\).
\end{cor}
\begin{proof}
  Let \(G_n\) be \(\alllamI \to (\forall x. p x = p(q x)) \to E_n\).
  Assume \(\HC(G_n) < 2_n^1/2\), then there exists \(N < 2_n^1 / 2\)
  and a Herbrand disjunction \(\bigvee_{i<N}(\lamI(\vec{t}_i) \to p
  u_i = p(q u_i) \to p q = p(T_n q q))\), namely, \((\bigwedge_{i<N}
  \lamI(\vec{t}_i)) \to (\bigwedge_{i<N} p u_i = p(q u_i)) \to p q =
  p(T_n q q)\), where \eps-free terms \(\vec{t}_i, u_i\) are in normal
  form.  Due to Theorem~\ref{t:Statman}, the number of instances of
  \(p x = p(q x)\) is at least \(2_n^1/2\), hence \(\HC(G_n) \ge 2_n^1/2\).
\end{proof}

The analysis in this section tells us that the lower bound of the
Herbrand complexity is hyperexponential.  Further technical
investigations will lead to a better lower bound, so that the gap between
the upper and lower bounds gets smaller, and also the unneccesity
of the property degree may be clarified.

\section{Conclusion}
\label{s:concl}

We studied the Herbrand complexity for epsilon calculus with the
\eps-equality axiom.  For Herbrand's theorem, which is on the prenex
existential formula, the Herbrand complexity is same as the case
without the \eps-equality formula, because the embedding result does
not rely on the \eps-equality formula.  The formulation of the
\eps-equality formula has to be restricted through the notion of
\eps-matrices, since otherwise as Yukami's trick explains, we would
fail to find the complexity bound.  Our proof of first epsilon theorem
is simpler than the original one by Bernays, as our formulation of the
\eps-equality formula due to the vector notation allows us to get rid
of the notion of closures from the proof.  Using this formulation, we
computed the upper bound of the Herbrand complexity for the first
extended epsilon theorem with the \eps-equality axiom.  While the
hyperexponential part of the result, namely, the height of the tower,
is the same as the case without the \eps-equality axiom, the
exponential part is quadratic with the additional parameters, the
property degree and the maximal arity, rather than the linear in the
case without the \eps-equality axiom.  Employing the original
formulation of the \eps-equality axiom by Bernays and Hilbert, the
parameter for the maximal arity can be got rid of, although the height
of the hyperexponential tower grows faster than the original result.
We also gave a lower bounds analysis which tells us that it has to be
at least hyperexponential, namely, non-elementary.

\paragraph{Future work}

There are open problems for future work.  The notion of property
degree was introduced to compute the upper bounds of the complexity in
Section~\ref{s:extended-epsilon-calculus-complexity-analysis}.  On the
other hand our lower bound analysis does not count the property
degree, hence it is still not clarified whether the property degree is
necessary for complexity analyses or not.  On the other hand, it is
important to explore a better proof representation suitable for
epsilon calculus, because a syntactic complication of \eps-terms is a
real practical obstacle to study epsilon calculus.  Some modern
formalizations of epsilon calculus are known, for example sequent
calculus with function variables by Baaz, Leitsch, and
Lolic~\cite{BaazLeitschLolic18} and another formulation based on
Miller's expansion tree by Aschieri, Hetzl, and
Weller~\cite{AschieriHetzlWeller18}.

\bibliographystyle{alpha}
\bibliography{eps-miyamoto,epsilon}

\end{document}